\documentclass[11pt]{article}
\usepackage{graphicx}
\usepackage{subcaption}
\usepackage{bmpsize}

\usepackage{amsmath}
\usepackage{amssymb}
\usepackage{amsthm}
\usepackage{amsfonts}
\usepackage{color}
\usepackage{algorithm,algorithmic}

\usepackage{url}
\usepackage{indentfirst}

\usepackage{kbordermatrix}
\usepackage{multirow}
\usepackage{dcolumn}

\usepackage{booktabs}

\usepackage[noadjust]{cite}  

\graphicspath{{./Figures/}}

\def\wht{\widehat}

\def\lam{\lambda}
\def\Lam{\Lambda}

\def\R{{\mathbb R}}
\def\C{{\mathbb C}}

\DeclareMathOperator{\diag}{diag}
\DeclareMathOperator{\rank}{rank}

\newtheorem{theorem}{\textbf{Theorem}}
\newtheorem{lemma}{\textbf{Lemma}}
\newtheorem{corollary}{\textbf{Corollary}}

\newtheorem{remark}{\rm{\textbf{Remark}}}
\newtheorem{example}{\rm{\textbf{Example}}}

\numberwithin{equation}{section}
\numberwithin{figure}{section}
\numberwithin{table}{section}

\newcommand{\onebytwo}[2]{
    \left[  \begin{array}{cc}
         #1 & #2
        \end{array} \right] }

\newcommand{\twobyone}[2]{
       \left[  \begin{array}{c}
         #1 \\
         #2
        \end{array} \right] }

\newcommand{\threebyone}[3]{
       \left[  \begin{array}{c}
         #1 \\
         #2 \\ 
         #3 
        \end{array} \right] }

\newcommand{\twobytwo}[4]{
       \left[ \begin{array}{cc}
        #1 & #2  \\
        #3 & #4
           \end{array} \right] }

\newcommand{\ignore}[1]{}

\newcommand{\mc}[1]{\mathcal{#1}}
\newcommand{\ldlt}{LDL$^{\rm T}$ } 
\newcommand{\MEDSKIP}{\medskip}

\newcolumntype{d}[1]{D{.}{.}{#1}}
\title{
On the shift-invert Lanczos method for  \\
the buckling eigenvalue problem 
}

\author{
Chao-Ping Lin
\thanks{Department of Mathematics, University of California, 
Davis, CA 95616, USA. (cplin@ucdavis.edu).}
\and
Huiqing Xie
\thanks{Department of Mathematics, East China University of Science and Technology,
Shanghai 200237, China. (hqxie@ecust.edu.cn).}
\and
Roger Grimes
\thanks{Livermore Software Technology Corporation, Livermore, CA 94551.}
\and
Zhaojun Bai
\thanks{Department of Computer Science, University of California, 
Davis, CA 95616, USA. (bai@cs.ucdavis.edu).}
}

\date{\today}

\begin{document}
\maketitle

\begin{abstract}
We consider the problem of extracting a few desired eigenpairs of 
the buckling eigenvalue problem $Kx = \lambda K_Gx$, 
where $K$ is symmetric positive semi-definite,
$K_G$ is symmetric indefinite, and 
the pencil $K - \lambda K_G$ is singular, namely, 
$K$ and $K_G$ share a non-trivial common nullspace. 
Moreover, in practical buckling analysis of structures, 
bases for the nullspace of $K$ and the common nullspace of $K$ 
and $K_G$ are available.  
There are two open issues for developing an industrial 
strength shift-invert Lanczos method: (1)
the shift-invert operator $(K - \sigma K_G)^{-1}$ 
does not exist or is extremely ill-conditioned, 
and (2) the use of the semi-inner product induced by $K$ 
drives the Lanczos vectors rapidly towards the nullspace
of $K$, which leads to a rapid growth of the Lanczos vectors in norms 
and cause permanent loss of information and the failure of the method.
In this paper, we address these two issues by proposing 
a generalized buckling spectral transformation of 
the singular pencil $K - \lam K_G$ and 
a regularization of the inner product via a low-rank updating 
of the semi-positive definiteness of $K$.
The efficacy of our approach is demonstrated by numerical examples, 
including one from industrial buckling analysis.  

\MEDSKIP
{\bf Keywords:} Eigenvalue problem, buckling analysis, Lanczos method, singular pencil

{\bf Mathematics Subject Classifications}: 65F15, 15A18
\end{abstract}


%


\section{Introduction}

We consider the buckling eigenvalue problem
\begin{align}  \label{eq:prob}
Kx = \lam K_Gx,
\end{align}
where $K$ and $K_G$ are $n\times n$ symmetric matrices, 
and $K$ is positive semi-definite and $K_G$ is indefinite. 
Furthermore, the pencil $K - \lam K_G$ is singular, 
i.e., the matrices $K$ and $K_G$ share a nontrivial 
common nullspace $\mc{Z}_c$. We are interested in 
(i) extracting a few nonzero finite eigenvalues around
a prescribed shift $\sigma\neq 0$ and
the associated eigenvectors $x$ perpendicular to 
the common nullspace $\mc{Z}_c$, 
and (ii) counting the number of eigenvalues of $K-\lambda K_G$ 
in a given interval $(\alpha, \beta)$.
As in practical buckling analysis of structures, 
we assume that a basis $Z\equiv [Z_N\ Z_C]$ of the nullspace of $K$ and 
a basis $Z_C$ of the common nullspace $\mc{Z}_c$ of 
$K$ and $K_G$ are available, and the pencil $K - \lam K_G$ is 
simultaneously diagonalizable.

The buckling eigenvalue problem \eqref{eq:prob} arises from 
the buckling analysis in structural engineering, 
where $K$ is referred to as the stiffness matrix and
$K_G$ is referred to as the geometric stiffness matrix.
The eigenvalue $\lam$ is used to determine the critical load
at which a structure may become unstable \cite[p.~72]{Komzsik2016},
and the eigenvector $x$ is the associated buckling shape.
The bases for the nullspace of $K$ and the common nullspace 
$\mc{Z}_c$ of $K$ and $K_G$ can be extracted
from the algebraic or geometric structure of 
the problem \cite{Farhat1998,Papa2001}.

The buckling eigenvalue problem \eqref{eq:prob} remains
an outstanding computational challenge in numerical linear algebra
\cite{Meerbergen2001,Stewart2009} and 
in industrial applications \cite{Grimes2016}.  
When the pencil $K - \lam K_G$ is regular and $K$ is positive definite,
a common practice for computing eigenpairs around a given shift 
$\sigma$ is  to convert \eqref{eq:prob} into the following ordinary 
eigenproblem
via a so-called {buckling} spectral transformation
\begin{align} \label{eq:bucklingTr}
(K- {\sigma} K_G)^{-1} K x = \frac{\lam}{{\lam} - {\sigma}} x,
\end{align}
see \cite{Ericsson1980,Nour-Omid1987,Grimes1994,Lehoucq1998}.
Since 
$(K- {\sigma} K_G)^{-1} K$ is symmetric with respect to $K$, 
the Lanczos method with $K$-inner product 
can be immediately used to solve the eigenproblem \eqref{eq:bucklingTr}. 
This approach is referred to as the shift-invert Lanczos method and
has been widely used, including in a number of industrial strength
eigensolvers, such as LS-DYNA  \cite{Grimes2016}. 

However, when $K$ is positive semi-definite and 
$K- {\lambda} K_G$ is singular, we have the following two issues: 
\begin{enumerate} 
\item 
Since the pencil $K-\lam K_G$ is {singular} or {near singular}, i.e., 
the matrices $K$ and $K_G$ share a non-trivial common nullspace $\mc{Z}_c$,
the shift-invert matrix $(K - \sigma K_G)^{-1}$ does not exist or 
is extremely ill-conditioned.  

\item Since the matrix $K$ is positive semi-definite,
the inner product induced by $K$ causes 
the Lanczos vectors driven rapidly toward the nullspace of $K$
\cite{Nour-Omid1987,Meerbergen1997,Meerbergen2001,Stewart2009}.
It results in the large norms of the Lanczos vectors,
which introduces large rounding errors. 
The accuracy of the computed solutions is degraded and even failed.  
\end{enumerate} 
These issues have been studied since the early development of
the shift-invert Lanczos method in the 1980s. 
Nour-Omid et al.  \cite{Nour-Omid1987} proposed a modified formulation of 
the Ritz vectors to refine the computed solutions.
Meerbergen \cite{Meerbergen2001} proposed to control the norms of 
the Lanczos vectors by applying implicit restart \cite{Sorensen1992}.
More recently, Stewart \cite{Stewart2009} gave a detailed analysis to 
show that the loss of information caused by the growth of the Lanczos vectors 
is permanent.

In this paper, we address the two issues by first
proposing a generalized buckling spectral transformation of 
the singular pencil $K-\lam K_G$, and 
a reguarlization of the inner product via a low-rank updating
of the positive semi-definite matrix $K$. 
Then a shift-invert Lanczos method 
for the buckling eigenvalue problem \eqref{eq:prob} is developed.
We will discuss two implementations of the matrix-vector product
for the computational kernel of 
the shift-invert Lanczos method, 
and propose two ways to count the number of 
eigenvalues in a given interval $(\alpha,\beta)$ for validation.

The rest of the paper is organized as follows. 
In  \S\ref{sec:theory}, we first present a canonical form 
of the pencil $K-\lam K_G$, and propose a generalized 
buckling spectral transformation, and 
a regularization of the inner product. 
In \S\ref{sec:SILan}, we discuss the implementation of 
the shift-invert Lanczos method with 
the generalized buckling spectral transformation 
and the regularized inner product. 
In \S\ref{sec:inertias}, we discuss two ways to 
count the number of eigenvalues in an interval.
Efficacy of the proposed approach is demonstrated in \S\ref{sec:example}.
Concluding remarks are given in \S\ref{sec:conclusion}.

Following the convention of matrix computations,
we use the upper case letters for matrices and lower case letters for vectors.
In particular, we use $I_n$ for the identity matrix of dimension $n$ with
$e_j$ being the $j$th column. 
If not specified, the dimensions of matrices and vectors conform to
the dimensions used in the context.
$\cdot^T$ is for transpose, $\cdot^{\dagger}$ for pseudo-inverse,
$\|\cdot\|_1$ for $1$-norm, and $\|\cdot\|_2$ and $\|\cdot\|_F$ 
for $2$-norm and Frobenius norm, respectively.
Also, we use $A^{-T}$ for the inverse of the matrix $A^{T}$.
The range and the nullspace of a matrix $A$ are denoted by 
$\mc{R}(A)$ and $\mc{N}(A)$, respectively.
The direct sum of two subspaces $\mc{S}_1$ and $\mc{S}_2$ is denoted by
$\mc{S}_1\oplus\mc{S}_2$.
The orthogonal complement to a subspace $\mc{S}$ is denoted by $\mc{S}^{\perp}$
and the orthogonal projection onto a subspace $\mc{S}$ is denoted by $\mc{P}_{\mc{S}}$.
$\nu_{+}(S)$, $\nu_{-}(S)$ and $\nu_{0}(S)$ 
denote the positive, negative and zero eigenvalues of 
a symmetric matrix $S$, respectively.
Other notations will be explained as used.


\section{Theory}  \label{sec:theory}


\subsection{Canonical form}

We start with a canonical form of the pencil $K -  \lambda K_G$.
For the compactness of presentation, 
we interchange the roles of $K$ and $K_G$ in \eqref{eq:prob} 
and consider the reversal of the pencil $K - {\lam}K_G$, i.e., 
$K_G - {\lam}^{\#}K$. 

\begin{theorem} \label{thm:can}
For the pencil $K_G-{\lam}^{\#} K$, 
there exists a non-singular matrix $W\in \R^{n\times n}$ such that
\begin{equation} \label{eq:can}
W^T K_G W =  \kbordermatrix{
& n_1       			& n_2			& n_3  	\\
n_1 & {\Lam}^{\#}_1     &          			&		\\ 
n_2& & {\Lam}^{\#}_2 			&		\\
n_3 &                   		&                     		& 0}
\quad \mbox{and} \quad
W^T K W =  \kbordermatrix{
		& n_1       	& n_2		& n_3   	\\
n_1          	& I_{n_1}  	&          		&		\\
n_2 		&             	& 0  			&		\\
n_3		&		& 			& 0},
\end{equation}
where ${\Lam}^{\#}_1$ and ${\Lam}^{\#}_2$ are diagonal matrices with real diagonal 
entries, and ${\Lam}^{\#}_2$ is non-singular.
Furthermore, by conformally partitioning $W = [W_1,W_2,W_3]$, 
we have 
\begin{align}  \label{eq:can_constraint}
W_3^TW_1 = 0\quad\mbox{and}\quad W_3^TW_2 = 0,
\end{align}
\end{theorem} 
\begin{proof} 
see Appendix \ref{appendix:can}. 
\end{proof}

\begin{remark}  \label{remark:can}
{\rm By the canonical form   \eqref{eq:can}, we immediately know that 
(i) the columns of $W_3$ span the common nullspace $\mc{Z}_c$ of $K$
and $K_G$, and the columns of $[W_1\ W_2]$ span the orthogonal 
complement to $\mc{Z}_c$, i.e., $\mc{Z}_c^{\perp}$;
(ii) the columns of $W_1$ are eigenvectors  
associated with real finite eigenvalues
$({\Lam}^{\#}_1,I_{n_1})$ of the pencil 
$K_G - {\lam}^{\#}K$ and are perpendicular to $\mc{Z}_c$;
(iii) The columns of $W_2$ are eigenvectors  associated with an infinite eigenvalue 
$({\Lam}^{\#}_2,0)$ of the pencil $K_G - {\lam}^{\#}K$
and are perpendicular to $\mc{Z}_c$;
(iv) For $x\in\mc{Z}_c$, $({\lam}^{\#}, x)$ is an eigenpair of the pencil 
$K_G - {\lam}^{\#}K$
for any ${\lam}^{\#}\in\C$.
} 
\end{remark}


\subsection{Generalized buckling spectral transformation}
Mathematically, the generalized   
buckling spectral transformation 
of the singular pencil $K - \lam K_G$ is to 
replace the inverse in \eqref{eq:bucklingTr}
by the pseudo-inverse and leads to the
ordinary eigenvalue problem
\begin{align}  \label{eq:defC}
Cx  = \mu x 
\quad \mbox{and} \quad 
C = (K - \sigma K_G)^{\dagger}K, 
\end{align}
where $(K - \sigma K_G)^{\dagger}$ is the 
pseudo-inverse of the singular matrix $K - \sigma K_G$
\cite[p.~290]{Golub2013}.
Note that the non-zero real shift $\sigma$ cannot be an eigenvalue of the pencil $K-\lam K_G$.

The following theorem provides the relationship 
of non-trivial eigenpairs
between the original buckling eigenvalue problem \eqref{eq:prob}
and the ordinary eigenvalue problem \eqref{eq:defC}. 
  
\begin{theorem}  \label{thm:eigC}
$(\lam, x)$ is an eigenpair of the pencil $K - \lam K_G$
with non-zero finite eigenvalue $\lam$ and $x\in\mc{Z}_c^{\perp}$
if and only if $(\mu, x)$ is an eigenpair of 
the matrix $C$ in \eqref{eq:defC} with $\mu\neq 0$ and $\mu\neq 1$ 
and $x\in\mc{Z}_c^{\perp}$, 
where $\mu=\frac{\lam}{\lam - \sigma}$ and $\sigma \neq 0$. 
\end{theorem} 

Before proving Theorem~\ref{thm:eigC}, 
we use the canonical form \eqref{eq:can} 
to derive an eigenvalue decomposition of $C$ and prove 
the eigenvalue and eigenvector relations between
$K_G - \lambda^{\#} K$ and $C$.


\begin{lemma}  \label{lemma:eigC}
With the canonical form \eqref{eq:can} in Theorem \ref{thm:can},
an eigenvalue decomposition of the matrix $C$ defined in~\eqref{eq:defC} is given by
\begin{align}  \label{eq:eigC}
CW 
 = W \left[ \begin{array}{ccc}
	(I_{n_1} - \sigma {\Lam}^{\#}_1)^{-1} &  &  \\
	&  0  &  \\
	&  &  0
	\end{array} \right].  
\end{align}
\end{lemma}
\begin{proof}
Recall that, since the matrix $K - \sigma K_G$ is symmetric,
\begin{align}  \label{eq:geometry1}
\mc{R}(K - \sigma K_G) = \mc{N}(K - \sigma K_G)^{\perp} = \mc{Z}_c^{\perp}.
\end{align}
In addition, by the condition \eqref{eq:can_constraint} in the 
canonical form~\eqref{eq:can}, we have  
\begin{align}  \label{eq:geometry2}
\mc{R}(W_1)\oplus\mc{R}(W_2) = \mc{R}(W_3)^{\perp} = \mc{Z}_c^{\perp}.
\end{align}
Therefore, from~\eqref{eq:geometry1} and \eqref{eq:geometry2},
\begin{align}  \label{eq:geometry}
\mc{R}(K - \sigma K_G) = \mc{R}(W_1)\oplus\mc{R}(W_2) 
= \mc{R}(W_3)^{\perp} = \mc{Z}_c^{\perp}.
\end{align}
Now note that, from the canonical form \eqref{eq:can}, 
\begin{align}  \label{eq:canKsKG}
W^TKW =   	 
	\left[ \begin{array}{ccc}
	I_{n_1} &  & \\
	&  0  &  \\
	&  &  0
	\end{array} \right]
\quad\mbox{and}\quad
W^T(K - \sigma K_G)W =   	 
	\left[ \begin{array}{ccc}
	I_{n_1} - \sigma {\Lam}^{\#}_1 &  & \\
	&  - \sigma{\Lam}^{\#}_2  &  \\
	&  &  0
	\end{array} \right].
\end{align}
Therefore we have
\begin{align}  \label{eq:canmatch}
W^TKW = 
 	\left[ \begin{array}{ccc}
	I_{n_1} &  & \\
	&  0  &  \\
	&  &  0
	\end{array} \right]
= W^T(K - \sigma K_G)W   	 
	\left[ \begin{array}{ccc}
	(I_{n_1} - \sigma {\Lam}^{\#}_1)^{-1} &  & \\
	&  0  &  \\
	&  &  0
	\end{array} \right].
\end{align}
Left multiplying~\eqref{eq:canmatch} by $W^{-T}$, it yields that  
\begin{align} \label{eq:KsKGmap}
KW = (K - \sigma K_G)W   	 
	\left[ \begin{array}{ccc}
	(I_{n_1} - \sigma {\Lam}^{\#}_1)^{-1} &  & \\
	&  0  &  \\
	&  &  0
	\end{array} \right].
\end{align}
From the Moore-Penrose conditions \cite[p.~290]{Golub2013},
\begin{align}  \label{eq:MoorePenrose}
(K - \sigma K_G)^{\dagger}(K - \sigma K_G) 
= \mc{P}_{\mc{R}((K - \sigma K_G)^T)}
= \mc{P}_{\mc{R}(K - \sigma K_G)}, 
\end{align}
namely $(K - \sigma K_G)^{\dagger}(K - \sigma K_G)$ 
is an orthogonal projection onto 
${\mc{R}((K - \sigma K_G)^T)} = {\mc{R}(K - \sigma K_G)}$.
Therefore, from~\eqref{eq:geometry} and \eqref{eq:MoorePenrose},
\begin{align}  \label{eq:projection}
(K - \sigma K_G)^{\dagger}(K - \sigma K_G)W 
= W	\left[ \begin{array}{ccc}
	I_{n_1}  &  & \\
	&  I_{n_2}  &  \\
	&  &  0
	\end{array} \right].
\end{align}
Left multiplying~\eqref{eq:KsKGmap} by $(K - \sigma K_G)^{\dagger}$
and using~\eqref{eq:projection}, we have the eigenvalue decomposition 
\eqref{eq:eigC} of $C$. 
%
\end{proof}

\begin{lemma}  \label{lemma:specC} 
The matrix $C$ defined in~\eqref{eq:defC} has the following properties:
\begin{itemize}
\item[(i)] 
$({\lam}^{\#}, x)$ is an eigenpair of $K_G - {\lam}^{\#}K$ 
with non-zero finite ${\lam}^{\#}$ and $x\in\mc{Z}_c^{\perp}$
if and only if 
$(\mu, x)$ is an eigenpair of $C$ 
with $\mu\neq 0$ and $\mu\neq 1$ and $x\in\mc{Z}_c^{\perp}$, 
where $\mu=\frac{1}{1-\sigma {\lam}^{\#}}$.

\item[(ii)]
$({\lam}^{\#}, x)$ is an eigenpair of $K_G - {\lam}^{\#}K$ 
with ${\lam}^{\#} = 0$ and $x\in\mc{Z}_c^{\perp}$
if and only if
$(\mu, x)$ is an eigenpair of $C$
with $\mu =1$ and $x\in\mc{Z}_c^{\perp}$.

\item[(iii)]
$({\lam}^{\#}, x)$ is an eigenpair of $K_G - {\lam}^{\#}K$ 
with $|{\lam}^{\#}| = \infty$ and $x\in\mc{Z}_c^{\perp}$
if and only if 
$(\mu, x)$ is an eigenpair of $C$
with $\mu = 0$ and $x\in\mc{Z}_c^{\perp}$.

\item[(iv)]
If $x\in\mc{Z}_c$, $Cx = 0$.
\end{itemize}
\end{lemma}

\begin{proof} 
The lemma can be proved by 
comparing the eigenvalue decomposition \eqref{eq:eigC} of $C$
with the canonical form \eqref{eq:can} of $K_G - {\lam}^{\#}K$. 
Specifically, for (i) and (ii), 
recall that each column of $W_1$ is an eigenvector associated with
a real, finite eigenvalue ${\lam}^{\#}$ of the pencil 
$K_G - {\lam}^{\#}K$
and the eigenvector is perpendicular to the common nullspace $\mc{Z}_c$.
From \eqref{eq:eigC}, 
each column of $W_1$ is now an eigenvector associated with 
a non-zero, finite eigenvalue $\mu = (1-\sigma{\lam}^{\#})^{-1}$ of 
the eigenproblem \eqref{eq:defC}.

To show (iii), recall that each column of $W_2$ is an eigenvector associated with
an infinite eigenvalue of the pencil $K_G - {\lam}^{\#}K$
and the eigenvector is perpendicular to the common nullspace $\mc{Z}_c$.
From \eqref{eq:eigC}, 
each column of $W_2$ is now an eigenvector associated with
zero eigenvalue of the eigenproblem \eqref{eq:defC}.

Finally, for (iv), 
the common nullspace $\mc{Z}_c$ is spanned by the columns of $W_3$
and, from \eqref{eq:eigC}, we know that $Cx = 0$ if $x\in\mc{Z}_c$.
\end{proof}

\noindent {\em Proof of Theorem~\ref{thm:eigC}.} 
Note that $(\lam, x)$ is an eigenpair of $K - \lam K_G$ 
with non-zero finite eigenvalue $\lam$ and $x\in\mc{Z}_c^{\perp}$
if and only if 
$({\lam}^{\#}, x)$ is an eigenpair of $K_G - {\lam}^{\#} K$ 
with non-zero finite eigenvalue 
${\lam}^{\#}=\lam^{-1}$ and $x\in\mc{Z}_c^{\perp}$.
Also, from Lemma \ref{lemma:eigC}(i), 
we know that $({\lam}^{\#}, x)$ is an eigenpair of 
$K_G - {\lam}^{\#} K$ 
with non-zero finite eigenvalue ${\lam}^{\#}$ and $x\in\mc{Z}_c^{\perp}$
if and only if
$(\mu, x)$ is an eigenpair of the eigenvalue problem $Cx = \mu x$
with $\mu = \frac{1}{1 - \sigma{\lam}^{\#}}$, $\mu\neq 0$ and $\mu\neq 1$, 
and $x\in\mc{Z}_c^{\perp}$.
Therefore,
$(\lam, x)$ is an eigenpair of the pencil $K - \lam K_G$ 
with non-zero finite eigenvalue $\lam$ and $x\in\mc{Z}_c^{\perp}$
if and only if 
$(\mu, x)$ is an eigenpair of the eigenvalue problem $Cx = \mu x$ 
with $\mu=\frac{\lam}{\lam - \sigma}$, $\mu\neq 0$ and $\mu\neq 1$, 
and $x\in\mc{Z}_c^{\perp}$.
\hfill $\Box$

By Theorem \ref{thm:eigC}, 
near the shift $\sigma$, 
the eigenpairs $(\lam,x)$ of $K-\lam K_G$ with
non-zero finite eigenvalues $\lam$ and $x\in\mc{Z}_c^{\perp}$
are transformed into eigenpairs $(\mu,x)$ of $C$ with
non-zero eigenvalues $\mu$, which typically are well-separated, 
and those away from the shift $\sigma$ are transformed into
clustered eigenpairs $(\mu,x)$ of $C$ near unity as shown in
Figure~\ref{fig:bucklingTr}. 
We note that the eigenpairs $(\mu,x)$ with $\mu=0$ or $\mu=1$ 
are not the ones of interest. 
The eigenpairs $(1,x)$ correspond 
to eigenpairs of $K-\lam K_G$ with infinite eigenvalues and
the eigenpairs $(0,x)$ correspond to eigenpairs of 
$K-\lam K_G$ with $x\in\mc{N}(K)$. 

\begin{figure}[t]
\centering
\includegraphics[width=0.48\textwidth,trim=10mm 10mm 10mm 10mm,clip=true]
{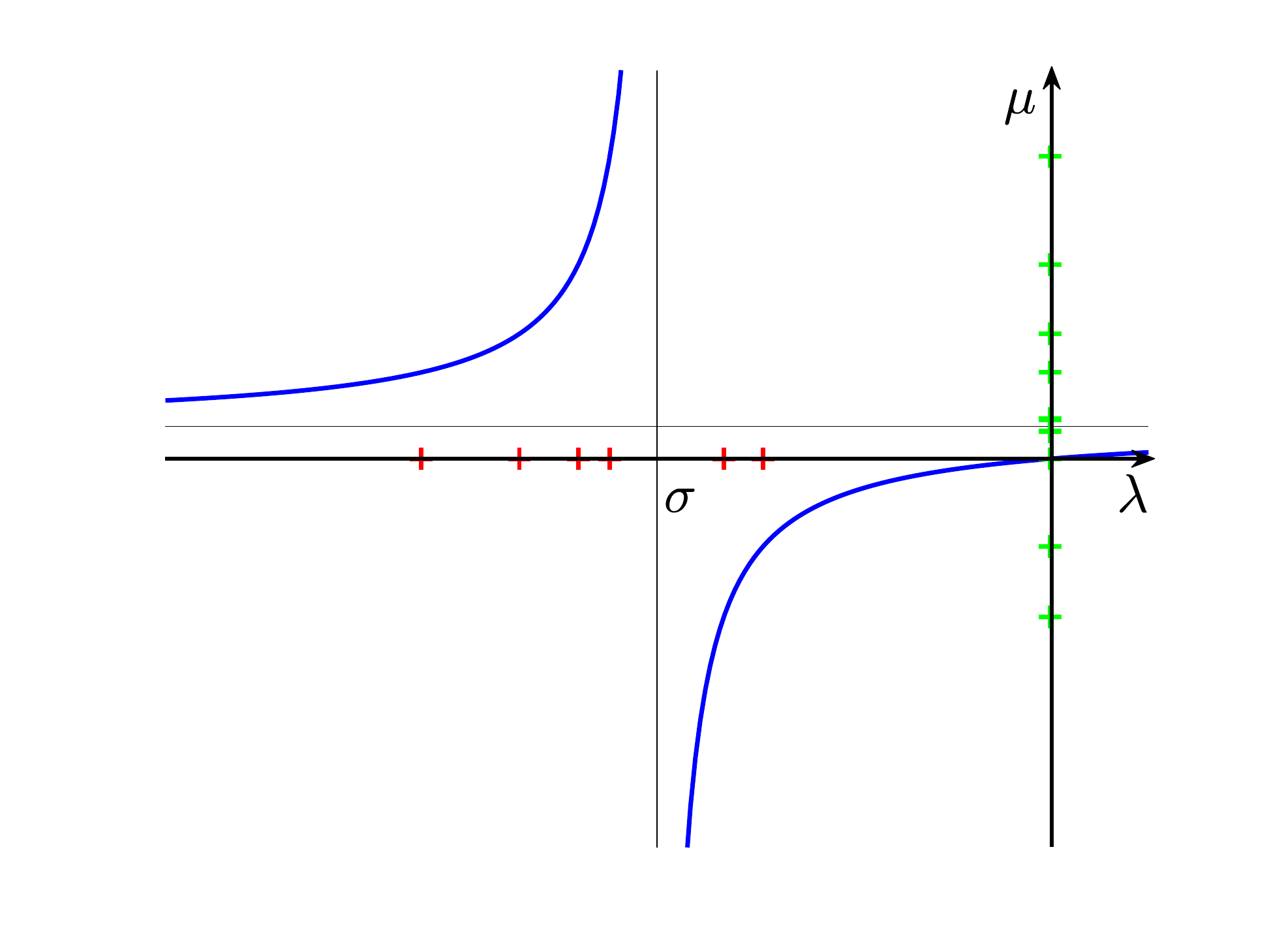}
\includegraphics[width=0.48\textwidth,trim=10mm 10mm 10mm 10mm,clip=true]
{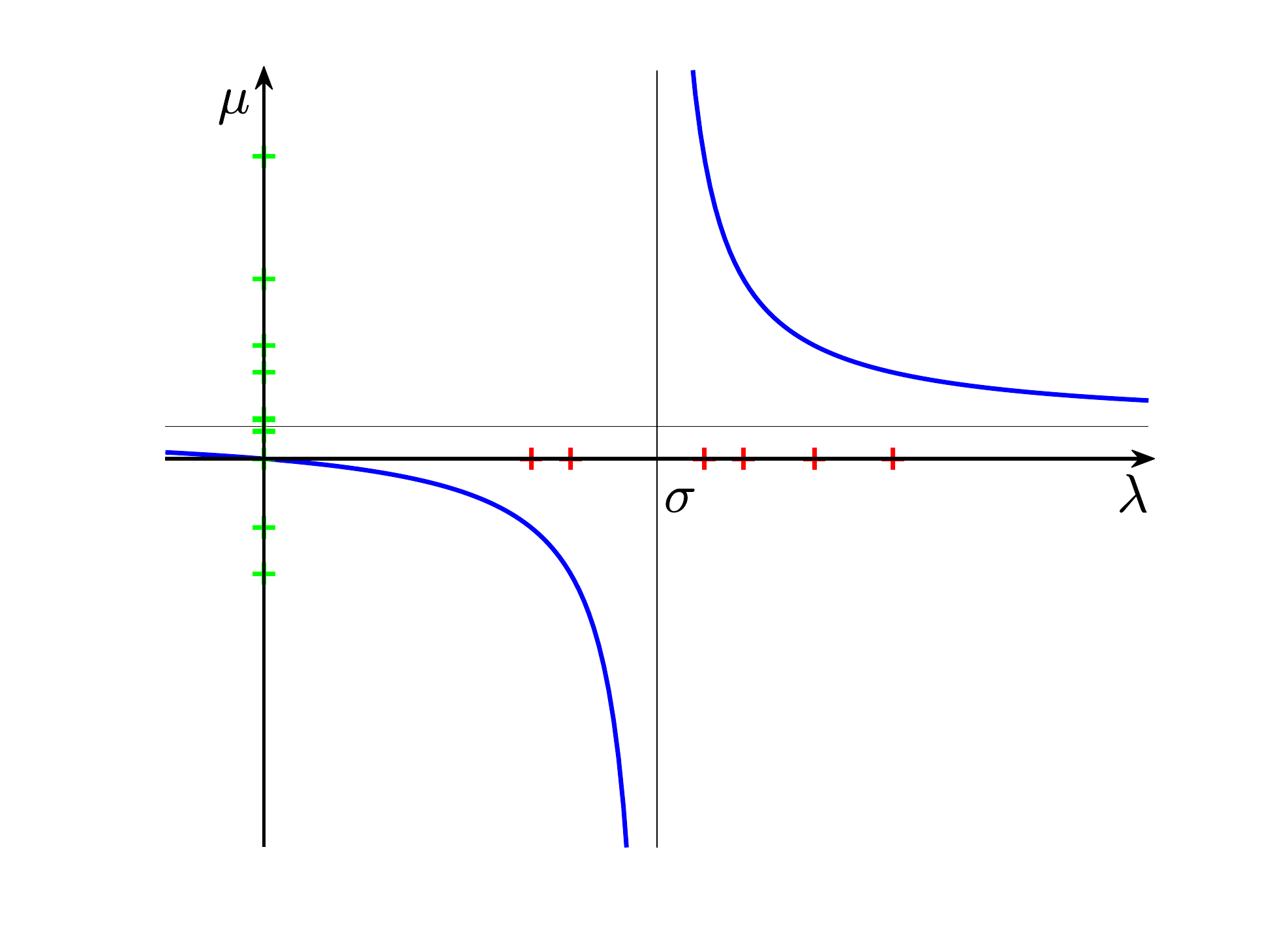}
\caption{Buckling spectral transfromation with 
$\sigma < 0$ (left) and $\sigma > 0$ (right).}
\label{fig:bucklingTr}
\end{figure}


\subsection{Regularization of the inner product}
In this subsection we introduce a positive definite matrix $M$ from 
a low-rank updating of $K$, and then show that the matrix $C$ 
in the generalized buckling spectral transformation~\eqref{eq:defC} 
is symmetric with respect to the inner product induced by $M$.

\begin{theorem}  \label{thm:lan}
Let $C$ be defined in \eqref{eq:defC}.
Let  $Z=[Z_N\ Z_C]$ span the nullspace $\mc{N}(K)$ 
and $Z_C$ span the common nullspace $\mc{Z}_c$ of $K$ and $K_G$. 
Define 
\begin{align}  \label{eq:defM}
M =  K + (K_GZ_N)H_N(K_GZ_N)^T + Z_CH_CZ_C^T,
\end{align}
where $H_N$ and $H_C$ are arbitrary positive definite matrices.
Then
\begin{itemize}
\item[$(i)$] 
the matrix $M$ is positive definite,

\item[$(ii)$] 
the matrix $C$ is symmetric with respect to 
the inner product induced by $M$.
\end{itemize}
\end{theorem}
\begin{proof}
By the canonical form \eqref{eq:can}, we have 
\begin{align*}
\mc{N}(K) = \mc{R}(W_2)\oplus\mc{R}(W_3) = \mc{R}(Z_N)\oplus\mc{R}(Z_C)
\quad\mbox{and}\quad
\mc{Z}_c = \mc{R}(W_3) = \mc{R}(Z_C),
\end{align*}
and 
\begin{align*}
\onebytwo{Z_N}{Z_C} = \onebytwo{W_2}{W_3}\twobytwo{R_{22}}{O}{R_{32}}{R_{33}}
\end{align*}
for some matrices $R_{22}\in\R^{n_2\times n_2}$, $R_{32}\in\R^{n_3\times n_2}$, 
$R_{33}\in\R^{n_3\times n_3}$, and $R_{22}$ and $R_{33}$ are non-singular.
Therefore, 
\begin{align*}
W^T K_GZ_N  = W^TK_G(W_2R_{22}+W_3R_{32})  
 = W^TK_GW_2R_{22}  
 = \threebyone{0}{{\Lam}^{\#}_2R_{22}}{0}.
\end{align*}
Since the basis $W$ satisfies the condition \eqref{eq:can_constraint},
\begin{align*}
W^TZ_C  = W^TW_3R_{33}  
 = \threebyone{0}{0}{(W_3^TW_3)R_{33}}.
\end{align*}
Therefore,
\begin{align} \label{eq:canM}
W^TMW  = W^T\big(K + (K_GZ_N)H_N(K_GZ_N)^T + Z_CH_CZ_C^T\big)W 
= \begin{bmatrix}
 I_{n_1} & & \\
 & \wht{H}_N &  \\
 & & \wht{H}_C
\end{bmatrix},
\end{align}
where 
$$
\wht{H}_N={\Lam}^{\#}_2R_{22}H_NR_{22}^T{\Lam}^{\#}_2
\quad\mbox{and}\quad
\wht{H}_C = (W_3^TW_3)R_{33}H_CR_{33}^T(W_3^TW_3).
$$
To prove that $M$ is positive definite, 
we show that both $\wht{H}_N$ and $\wht{H}_C$ are positive definite.
For the matrix $\wht{H}_N$, we note that
the matrix $H_N$ is positve definite 
and the matrix $R_{22}$ is non-singular.
Also, from Theorem \ref{thm:can}, 
the diagonal matrix ${\Lam}^{\#}_2$ is non-singular.
Therefore, the matrix $\wht{H}_N$ is positive definite. 
For the matrix $\wht{H}_C$, 
we note that the matrix $H_C$ is positive definite
and the matrix $R_{33}$ is non-singular.
Also, since the matrix $W_3$ is of full rank, 
the symmetric matrix $W_3^TW_3$ is non-singular.
Therefore, the matrix $\wht{H}_C$ is also positive definite. 
This proves $(i)$.

To prove $(ii)$, 
by the eigenvalue decomposition~\eqref{eq:eigC} of $C$ 
and~\eqref{eq:canM}, we have
\begin{align*}
W^TMCW = W^TM W W^{-1}CW = 
\left[ \begin{array}{ccc}
	(I_{n_1} - \sigma {\Lam}^{\#}_1)^{-1} &  & \\
	&  0  &  \\
	&  &  0
	\end{array} \right].
\end{align*}
Therefore, the matrix $MC$ is symmetric, which means that
the matrix $C$ is symmetric with respect to the inner product 
induced by $M$.
\end{proof}

\begin{remark}{\rm
We note that if 
the pencil $K - \lam K_G$ is regular, Theorem~\ref{thm:lan} is still 
applicable.  In this case, 
the matrix $C$ in \eqref{eq:bucklingTr} is symmetric with respect to
the inner product induced by 
$M = K + (K_GZ_N)H_N(K_GZ_N)^T$.
}\end{remark}


\section{Shift-invert Lanczos method}  \label{sec:SILan}

\subsection{Shift-invert Lanczos method}
By Theorem \ref{thm:eigC},
we have generalized the buckling spectral transformation to the singular pencil 
$K-\lam K_G$ and converted the buckling eigenproblem \eqref{eq:prob}
into an equivalent ordinary eigenvalue problem \eqref{eq:defC}.
From Theorem \ref{thm:lan}, we know that 
the matrix $C$ in \eqref{eq:defC} is symmetric with respect to
the inner product induced by the positive definite 
matrix $M$ in \eqref{eq:defM}. 
It naturally leads that 
to solve the buckling eigenvalue problem~\eqref{eq:prob}, 
we can use the Lanczos method on the matrix $C$ 
with the inner product induced by $M$.
This new strategy is also referred to as the shift-invert Lanczos method
and outlined in Algorithm~\ref{alg:LanFRO}.

The shift-invert Lanczos method, after $j$ steps, computes
a sequence of Lanczos vectors $\{v_1\ \ldots\ v_{j+1}\}$
and a symmetric tridiagonal matrix 
$T_j = \mbox{tridiag}( \beta_{i-1}, \alpha_{i}, \beta_{i})$
satisfying the governing equations
\begin{align} \label{eq:Lanczos} 
CV_j = V_jT_j + \beta_{j}v_{j+1}e_j^T
\quad\mbox{and}\quad 
V_{j+1}^TMV_{j+1} = I_{j+1},
\end{align}
where $V_{j+1} \equiv [v_1\ \ldots\ v_{j+1}]$.
Great cares must be taken to ensure that the equations in \eqref{eq:Lanczos} 
are satisfied \cite{Parlett1979,Ericsson1980,Simon1984,Nour-Omid1987,Grimes1994}
in the presence of finite-precision arithmetic. Several techniques have been 
developed and well-implemented \cite{Parlett1979,Simon1984,Grimes1994}.
For the rest of discussion, we will focus on the implementations of
the matrix-vector product $u=Cv$ for Line 6 of Algorithm~\ref{alg:LanFRO}.

\begin{algorithm}
\caption{Shift-invert Lanczos method for the buckling eigenvalue problem~\eqref{eq:prob}} 
\label{alg:LanFRO}
\begin{algorithmic}[1] 
\STATE{$r \leftarrow v$, where $v$ is the starting vector}
\STATE{$p \leftarrow Mr$, where $M =  K+(K_GZ_N)H_N(K_GZ_N)^T+Z_CH_CZ_C^T$ }
\STATE{$\beta_0 \leftarrow (p^Tr)^{1/2}$}
\FOR{$j = 1,2,\ldots$}
\STATE{$v_j \leftarrow r/\beta_{j-1}$}
\STATE{$r \leftarrow Cv_{j}$, where $C = (K-\sigma K_G)^{\dagger}K$ }
\STATE{$r \leftarrow r - \beta_{j-1}v_{j-1}$}
\STATE{$p \leftarrow Mr$}
\STATE{$\alpha_j \leftarrow v_{j}^Tp$}
\STATE{$r \leftarrow r - \alpha_jv_j$}
\STATE{perform re-orthogonalization if necessary}
\STATE{$p \leftarrow Mr$}
\STATE{$\beta_{j}\leftarrow (p^Tr)^{1/2}$}
\STATE{Compute the eigenvalue decomposition of $T_j$} 
\STATE{Check convergence}
\ENDFOR
\STATE{Compute approximate eigenvectors of the converged eigenpairs}
\end{algorithmic}
\end{algorithm}

\subsection{The matrix-vector product}  \label{sec:issues}

We first show that the matrix-vector product 
$u = Cv = (K - \sigma K_G)^{\dagger}Kv$
is connected
with the solution of a consistent singular linear system with constraint. 
Based on this connection, we present two ways for computing the vector $u$. 


\begin{theorem}  \label{thm:sinLin}
Given $v\in\R^n$, the vector
\begin{align}  \label{eq:MxV}
u =  (K - \sigma K_G)^{\dagger}Kv
\end{align}
is the unique solution of
the consistent singular linear system 
\begin{align}  \label{eq:sinLin}
(K - \sigma K_G)u = Kv
\end{align}
with the constraint 
\begin{align}  \label{eq:constraint}
Z_C^Tu = 0, 
\end{align}
where $Z_C$ is a basis of the common nullspace of $K$ and $K_G$. 
\end{theorem}
\begin{proof}
First note that since both $K$ and $K - \sigma K_G$ are symmetric, 
we have
\begin{align}  \label{eq:RNperp}
\mc{R}(K)=\mc{N}(K)^{\perp}
\quad\mbox{and}\quad 
\mc{R}(K - \sigma K_G)=\mc{N}(K - \sigma K_G)^{\perp}=\mc{Z}_c^{\perp}
\end{align}
and 
\begin{align}  \label{eq:Nsubset}
\mc{Z}_c = \mc{N}(K - \sigma K_G)\subset\mc{N}(K).
\end{align}
Therefore from~\eqref{eq:RNperp} and \eqref{eq:Nsubset},
\begin{align*} 
Kv\in\mc{R}(K)\subset\mc{R}(K - \sigma K_G), 
\end{align*}
which implies that the linear system~\eqref{eq:sinLin} is consistent.
From~\eqref{eq:MxV},
\begin{align}  \label{eq:sinLin2} 
(K - \sigma K_G)u 
=  (K - \sigma K_G)(K - \sigma K_G)^{\dagger}Kv  
= \mc{P}_{\mc{R}(K - \sigma K_G)}Kv 
= Kv,  
\end{align}
where $\mc{P}_{\mc{R}(K - \sigma K_G)}$ is an orthogonal projection 
onto ${\mc{R}(K - \sigma K_G)}$ 
(by the Moore-Penrose conditions \cite[p.~290]{Golub2013}). 
This means that $u$ is a solution of 
the consistent singular linear system \eqref{eq:sinLin}.

On the other hand, from~\eqref{eq:MxV} and \eqref{eq:sinLin2},
\begin{align} \label{eq:Pu1}
u = (K - \sigma K_G)^{\dagger} Kv
   = (K - \sigma K_G)^{\dagger}(K - \sigma K_G)u
   = \mc{P}_{\mc{R}((K - \sigma K_G)^T)}u
   = \mc{P}_{\mc{R}(K - \sigma K_G)}u.
\end{align}
Since $\mc{R}(K - \sigma K_G) = \mc{Z}_c^{\perp}$,
it implies that $u$ is perpendicular to the common nullspace $\mc{Z}_c$,
which is also the nullspace $\mc{N}(K - \sigma K_G)$. 

The uniqueness can be shown as follows.
Given two solutions $u_1$ and $u_2$ to~\eqref{eq:sinLin},
the difference $u_1-u_2$ would satisfy $(K-\sigma K_G)(u_1-u_2) = 0$,
which implies $u_1-u_2\in\mc{Z}_c$. However, 
since both solutions satisfy the constraint \eqref{eq:constraint},
$Z_C^T(u_1-u_2) = 0$. Therefore, $u_1-u_2 = 0$.
\end{proof}

We now present two methods to compute the matrix-vector product $u = Cv$.  

\paragraph{Method 1.}
By Theorem \ref{thm:sinLin}, a straightforward method 
is to solve the augmented linear system 
\begin{align}  \label{eq:augLin}
A_{\sigma} \twobyone{u}{v} = \twobyone{Kv}{0}
\quad \mbox{and} \quad
A_{\sigma} = \twobytwo{K - \sigma K_G}{Z_C}{Z_C^T}{0}.
\end{align}
The system \eqref{eq:augLin} is nonsingular, and 
$[u^T\ 0]^T$ is a unique solution. 
This is due to 
the fact that if we  consider the corresponding homogeneous system  
\begin{align}  \label{eq:augLin_foot}
A_{\sigma} y = \twobytwo{K - \sigma K_G}{Z_C}{Z_C^T}{0}
	\twobyone{y_1}{y_2} = \twobyone{0}{0}, 
\end{align}
by $\R^n=\mc{R}(K - \sigma K_G)\oplus\mc{R}(Z_C)$, 
the first block row of \eqref{eq:augLin_foot} 
leads to $(K-\sigma K_G)y_1=Z_Cy_2=0$.
Since $Z_C$ is of full rank, $y_2=0$.
Therefore we have $(K-\sigma K_G)y_1=0$ and $Z_C^Ty_1=0$,
it implies $y_1=0$. 

We note that the linear system of the form \eqref{eq:augLin} 
appears in various applications \cite{Benzi2005}. 
Recent advances include \cite{Ashcraft1998,Duff2017} on direct methods 
and \cite{Estrin2018} for iterative methods. 

\paragraph{Method 2.}
Note that the leading principal submatrix $K - \sigma K_G$ of $A_{\sigma}$ 
in \eqref{eq:augLin} is singular.
{The pivoting during the sparse \ldlt factorization of $A_{\sigma}$
would result in a permutation matrix which interchanges the rows 
in (1,1)-block of $A_{\sigma}$ with the basis $Z_C$ in (2,1)-block. 
When $Z_C$ is dense,
a significant number of fill-ins in the lower triangular matrix $L$ occurs
(see Example 2 in Section \ref{sec:example}).}
To circumvent this, we consider an alternative strategy as follows. 
First, we have the following theorem to extract a non-singular 
submatrix of $K - \sigma K_G$ by exploiting the basis $Z_C$.
\begin{theorem} \label{thm:rrp}
Let $Z_C\in\R^{n\times n_3}$ be a basis of $\mc{N}(K - \sigma K_G)$
and $P\in\mathbb{R}^{n\times n}$ be a permutation matrix
such that
$P^TZ_C\equiv\twobyone{Y_1}{Y_2}$,
and $Y_2\in\mathbb{R}^{n_3\times n_3}$ is non-singular.
Define 
\begin{align}  \label{eq:rrp}
S =  P^T(K - \sigma K_G)P 
\quad \mbox{and} \quad  
S = \kbordermatrix{& n-n_3 & n_3  \\
n-n_3 & S^{\sigma}_{11} & S_{12}  \\
n_3    & S_{12}^T & S_{22}}.
\end{align}
Then 
\begin{itemize} 
\item[(1)] the submatrix $S^{\sigma}_{11}\in\R^{(n-n_3)\times(n-n_3)}$ 
is non-singular,
\item[(2)]  
$\nu_{+}(S^{\sigma}_{11}) = \nu_{+}(K - \sigma K_G)$ 
and 
$\nu_{-}(S^{\sigma}_{11}) = \nu_{-}(K - \sigma K_G)$, 
where $\nu_{+}(X)$ and $\nu_{-}(X)$
denote the numbers of positive and negative eigenvalues of the symmetric matrix $X$, respectively.
\end{itemize} 
\end{theorem}
\begin{proof}
Let
\begin{align*}  
E = \kbordermatrix{& n-n_3 & n_3  \\
n-n_3 & I_{n-n_3} & Y_1  \\
n_3    & 0 & Y_2}
\in\mathbb{R}^{n\times n}.
\end{align*}
The matrix $E$ is non-singular since $Y_2$ is non-singular.
By the congruence transformation, we have 
\begin{align}  \label{eq:revealing}
E^T S E = 
E^T P^T(K - \sigma K_G)P E = 
E^T \twobytwo{S^{\sigma}_{11}}{S_{12}}{S_{12}^T}{S_{22}} E
= \kbordermatrix{& n-n_3 & n_3  \\
n-n_3 & S^{\sigma}_{11} & 0  \\
n_3    & 0 & 0}. 
\end{align}
Sylvester's law \cite[p.~448]{Golub2013} tells that 
the matrices $K-\sigma K_G$ and $E^TSE$ 
have the same inertias. In particular, from~\eqref{eq:revealing}, 
we know that 
\begin{align*}
\nu_{+}(K-\sigma K_G)=\nu_{+}(S_{11}^{\sigma}),\quad
\nu_{-}(K-\sigma K_G)=\nu_{-}(S_{11}^{\sigma}),
\end{align*} 
and
\begin{align}  \label{eq:nu0}
\nu_{0}(K-\sigma K_G) = \nu_{0}(S_{11}^{\sigma}) + n_3
\end{align} But $\nu_{0}(K-\sigma K_G) = \dim(\mc{N}(K-\sigma K_G)) = n_3$.
Therefore, from~\eqref{eq:nu0}, 
$\nu_{0}(S_{11}^{\sigma}) = 0$ and $S_{11}^{\sigma}$ is non-singular.
\end{proof}

Theorem~\ref{thm:rrp}
was inspired by \cite[Theorem~2.2]{Arbenz2002}
where the authors consider solving
a consistent semi-definite linear systems $Ax=b$
from the electromagnetic applications \cite{Arbenz2001}.
The matrix $A$, generated from the finite element modeling,
is positive semi-definite and an explicit basis of
the nullspace of $A$ is available. This explicit basis of 
the nullspace is then used to identify a non-singular part 
of $A$ and a solution of the linear system can be computed from it.
Although in the buckling eigenvalue probem~\eqref{eq:prob},
the matrix $K-\sigma K_G$ is indefinite,
we found that the strategy developed in \cite{Arbenz2002}
can be generalized to the system \eqref{eq:sinLin} and 
\eqref{eq:constraint}.
By this strategy, the fill-ins of the lower triangular matrix $L$
can be significantly reduced,  see Example 2 in \S\ref{sec:example}.

By Theorem~\ref{thm:rrp}, an alternative method to solve \eqref{eq:sinLin} 
can be described in two steps:
\begin{enumerate}
\item Find a solution $u_p$ of the consistent singular linear system 
\eqref{eq:sinLin}.
\item Compute $u = \mc{P}_{\mc{R}(K - \sigma K_G)}u_p$
to satisfy the constraint  \eqref{eq:constraint},
where $\mc{P}_{\mc{R}(K - \sigma K_G)}$ is an orthogonal projection
onto $\mc{R}(K - \sigma K_G)$.
\end{enumerate}
Specifically, in Step 1, 
find the permutation matrix $P$ as described in Theorem \ref{thm:rrp},
and rewrite~\eqref{eq:sinLin} in the partitioned form \eqref{eq:rrp}:
\begin{align}  \label{eq:S}
\twobytwo{S^{\sigma}_{11}}{S_{12}}{S_{12}^T}{S_{22}}\twobyone{w_1}{w_2}
=\twobyone{c_1}{c_2}\in\mc{R}(S),
\end{align}
where
\begin{align*}
\twobyone{w_1}{w_2}\equiv P^Tu
\quad\mbox{and}\quad
\twobyone{c_1}{c_2}\equiv P^T Kv.
\end{align*}
Since $S^{\sigma}_{11}$ is non-singular, $S^{\sigma}_{11}$ is of full rank and
the leading $n-n_3$ columns of $S$ are linearly independent.
On the other hand, we know that $\rank(S)=\rank(K - \sigma K_G)=n-n_3$. 
Therefore, the leading $n-n_3$ columns of $S$ is a basis of $\mc{R}(S)$,
and there is a solution $w_p$ of~\eqref{eq:S} with $w_2 = 0$.
Direct substitution gives
\begin{align*}
w_p = \twobyone{(S^{\sigma}_{11})^{-1}c_1}{0},
\end{align*}
where the inverse $(S^{\sigma}_{11})^{-1}$ can be computed 
using the sparse \ldlt  factorization of $S^{\sigma}_{11}$
\cite{Ashcraft1998,Duff2017}. 
A solution $u_p$ of~\eqref{eq:sinLin} is then given by 
\begin{align*}
u_p = P\twobyone{(S^{\sigma}_{11})^{-1}c_1}{0}.
\end{align*}
In Step 2,  since $Z_C$ is a basis of $\mc{N}(K - \sigma K_G)$,
which is the orthogonal complement to $\mc{R}(K - \sigma K_G)$,
the vector $u$ can be computed by the projection
\begin{align*}
u = \mc{P}_{\mc{R}(K - \sigma K_G)}u_p
 = (I - Z_C (Z^T_C Z_C)^{-1} Z_C^T)u_p.
\end{align*}
If $Z_C$ is an orthonormal basis, then
\begin{align*}
u = \mc{P}_{\mc{R}(K - \sigma K_G)}u_p
 = (I - Z_C Z_C^T)u_p.
\end{align*}


\section{Counting eigenvalues}
\label{sec:inertias}
In this section,
as a validation scheme, we discuss ways to count the number of eigenvalues in 
a given interval. 
In the following, 
$\nu_{+}(A)$ and $\nu_{-}(A)$ 
denote the number of positive and negative eigenvalues of
a symmetric matrix $A$, respectively. 
$n(\alpha,\beta)$ and ${n}^{\#}(\alpha,\beta)$ 
denote the numbers of eigenvalues of 
the pencil $K-\lam K_G$ and the reversed pencil $K_G-{\lam}^{\#}K$ 
in an interval $(\alpha,\beta)$, respectively.

First, we consider the following lemma.
\begin{lemma}  \label{lemma:countEig}
Let  $Z=[Z_N\ Z_C]$ span the nullspace $\mc{N}(K)$
and $Z_C$ span the common nullspace $\mc{Z}_c$ of $K$ and $K_G$, then 
\begin{itemize}
\item[$(i)$]
for $\alpha < 0$, 
$n(\alpha,0) = \nu_{-}(K - \alpha K_G) - \nu_{-}(Z_N^TK_GZ_N)$, 

\item[$(ii)$]
for $\alpha > 0$,
$n(0,\alpha) = \nu_{-}(K - \alpha K_G) - \nu_{+}(Z_N^TK_GZ_N).$
\end{itemize}
In addition, the matrix $Z_N^TK_GZ_N$ is non-singular.
\end{lemma}
\begin{proof} The proof is based on the following two facts: 
(1) $(\lam, x)$ is an eigenpair of the pencil $K - \lam K_G$ 
with non-zero finite eigenvalue $\lam$ and $x\in\mc{Z}_c^{\perp}$
if and only if 
$({\lam}^{\#}, x)$ is an eigenpair of the pencil $K_G - {\lam}^{\#} K$ 
with non-zero finite eigenvalue ${\lam}^{\#}=\lam^{-1}$ and 
$x\in\mc{Z}_c^{\perp}$.
(2) By the canonical form \eqref{eq:can}, we have
\begin{align*}
W^T(K_G - \frac{1}{\alpha}K)W = 
        \left[ \begin{array}{ccc}
        {\Lam}^{\#}_1 - \frac{1}{\alpha}I_{n_1} &       &       \\
        &       {\Lam}^{\#}_2   &       \\
        &       &       0
        \end{array} \right]. 
\end{align*}
Consequently, by Sylvester's law, we have 
\begin{align*}  
\nu_{-}(K_G - \frac{1}{\alpha}K) 
& = \nu_{-}({\Lam}^{\#}_1 - \frac{1}{\alpha}I_{n_1}) + \nu_{-}({\Lam}^{\#}_2), \\
\nu_{+}(K_G - \frac{1}{\alpha}K) 
& = \nu_{+}({\Lam}^{\#}_1 - \frac{1}{\alpha}I_{n_1}) + \nu_{+}({\Lam}^{\#}_2).
\end{align*}

Now, for (i), since $\alpha < 0$,
\begin{align}  
n(\alpha,0) & ={n}^{\#}(-\infty,\frac{1}{\alpha})
 = \nu_{-}({\Lam}^{\#}_1-\frac{1}{\alpha}I_{n_1}) 
 = \nu_{-}(K_G - \frac{1}{\alpha}K) - \nu_{-}({\Lam}^{\#}_2) 
 = \nu_{-}(K - \alpha K_G) - \nu_{-}({\Lam}^{\#}_2)  \label{eq:countEig1},
\end{align} 
where for the second equality, see Remark \ref{remark:can}. 
For (ii), since 
$\alpha > 0$, 
\begin{align}  
n(0,\alpha)
 ={n}^{\#}(\frac{1}{\alpha},+\infty) 
 =\nu_{+}({\Lam}^{\#}_1-\frac{1}{\alpha}I_{n_1}) 
 = \nu_{+}(K_G - \frac{1}{\alpha}K) - \nu_{+}({\Lam}^{\#}_2) 
 = \nu_{-}(K - \alpha K_G) - \nu_{+}({\Lam}^{\#}_2)  \label{eq:countEig2}.
\end{align} 
On the other hand, by the canonical form \eqref{eq:can}, we have 
\begin{align*}
\mc{N}(K) = \mc{R}(Z_N)\oplus\mc{R}(Z_C) = \mc{R}(W_2)\oplus\mc{R}(W_3)
\quad\mbox{and}\quad
\mc{Z}_c = \mc{R}(Z_C) = \mc{R}(W_3),
\end{align*}
and 
\begin{align*}
Z_N = W_2R_{22} + W_3R_{32},
\end{align*}
where $R_{22}\in\R^{n_2\times n_2}$, $R_{32}\in\R^{n_3\times n_2}$ 
and $R_{22}$ is non-singular.
Also, we know that $W_2^TK_GW_2 = {\Lam}^{\#}_2$.
Therefore,
\begin{align*}
Z_N^TK_GZ_N = R_{22}^T(W_2^TK_GW_2)R_{22} = R_{22}^T{\Lam}^{\#}_2R_{22}.
\end{align*}
This implies that the matrix $Z_N^TK_GZ_N$ is non-singular, 
and by Sylvester's law, we have 
\begin{align}  \label{eq:Lam2}
\nu_{-}({\Lam}^{\#}_2) = \nu_{-}(Z_N^TK_GZ_N)
\quad\mbox{and}\quad
\nu_{+}({\Lam}^{\#}_2) = \nu_{+}(Z_N^TK_GZ_N).
\end{align}
The lemma is an immediate consequence 
of~\eqref{eq:countEig1}, \eqref{eq:countEig2} and \eqref{eq:Lam2}. 
\end{proof}

Lemma \ref{lemma:countEig} establishes the relation between
the number of eigenvalues in the interval
$(\alpha, 0)$ or $(0, \alpha)$ and the inertia $\nu_{-}(K-\alpha K_G)$.
Below, we discuss how to express the inertia $\nu_{-}(K-\alpha K_G)$
in terms of the augmented matrix $A_{\alpha}$ in \eqref{eq:augLin} 
and the submatrix $S^{\alpha}_{11}$ in \eqref{eq:rrp}. 

\begin{lemma}  \label{lemma:countInertia}
In terms of the augmented matrix $A_{\alpha}$ in \eqref{eq:augLin} 
and the submatrix $S^{\alpha}_{11}$ in \eqref{eq:rrp}, 
\begin{align}  \label{eq:countInertia1}
\nu_{-}(K - \alpha K_G) = \nu_{-}(A_{\alpha}) - \dim(\mc{Z}_c)
 = \nu_{-}(S^{\alpha}_{11}).
\end{align}
\end{lemma}
\begin{proof}
Considering the singular value decomposition $Z_C=U\Sigma V^T$,
$\Sigma\in\R^{n_3\times n_3}$,
and the partial eigen-decomposition $(K - \alpha K_G)X = X\Lam$,
where $\Lam\in\R^{(n-n_3)\times(n-n_3)}$ is a diagonal matrix 
consisting of all the non-zero eigenvalues of $K - \alpha K_G$,
we can construct the following eigen-decomposition of $A_{\alpha}$:
\begin{align}  \label{eq:spectrumA}
A_{\alpha} 
\left[ \begin{array}{ccc}
X & \frac{1}{\sqrt{2}}U & \frac{1}{\sqrt{2}}U  \\
0 & \frac{1}{\sqrt{2}}V & -\frac{1}{\sqrt{2}}V
\end{array} \right]
& = \left[ \begin{array}{ccc}
X & \frac{1}{\sqrt{2}}U & \frac{1}{\sqrt{2}}U  \\
0 & \frac{1}{\sqrt{2}}V& -\frac{1}{\sqrt{2}}V
\end{array} \right]
\left[ \begin{array}{ccc}
\Lam &  &   \\
& \Sigma &  \\
& & -\Sigma
\end{array} \right].  
\end{align}
In \eqref{eq:spectrumA}, 
the diagonal entries of $\Sigma$ are positive since $Z_C$ is of full rank,
and the first equality of~\eqref{eq:countInertia1} is proved by
counting the number of negative eigenvalues of $A_{\alpha}$.
The second equality immediately follows from Theorem \ref{thm:rrp}.
\end{proof}

Combining Lemmas \ref{lemma:countEig} and \ref{lemma:countInertia},
we have the following theorem which provides 
a computational approach to count the number of eigenvalues 
of $K-\lambda K_G$ using
the inertias of $A_{\alpha}$ or $S^{\alpha}_{11}$. 

\begin{theorem} 
In terms of the augmented matrix $A_{\alpha}$ in \eqref{eq:augLin}
and the submatrix $S^{\alpha}_{11}$ in \eqref{eq:rrp}, we have 
\begin{itemize}
\item[$(i)$]
$n(\alpha,0)=\nu_{-}(A_{\alpha})-\dim(\mc{Z}_c)-\nu_{-}(Z_N^TK_GZ_N)
=\nu_{-}(S^{\alpha }_{11})-\nu_{-}(Z_N^TK_GZ_N)$, 
if $\alpha< 0$.

\item[$(ii)$]
$n(0,\alpha)=\nu_{-}(A_{\alpha})-\dim(\mc{Z}_c)-\nu_{+}(Z_N^TK_GZ_N)
=\nu_{-}(S^{\sigma}_{11})-\nu_{+}(Z_N^TK_GZ_N)$, 
if $\alpha > 0$.
\end{itemize}
\end{theorem}
\begin{remark}{\rm
In practice, the inertias $\nu_{-}(A_{\alpha})$ and $\nu_{-}(S_{11}^{\alpha})$
are by-products of the sparse \ldlt factorizations of the matrices 
$A_{\alpha}$ and $S_{11}^{\alpha}$, respectively \cite[p.~214]{Higham2002}.
The inertias $\nu_{-}(Z_N^TK_GZ_N)$ and $\nu_{+}(Z_N^TK_GZ_N)$ 
can be easily computed since the size of $Z_N^TK_GZ_N$ is typically small
in buckling analysis. 
}
\end{remark}


\section{Numerical examples}  \label{sec:example}
In this section, we first use a synthetic example 
to illustrate the growth of the norms of the Lanczos vectors 
with $K$-inner product and the consequence of the growth
as discussed by Meerbergen \cite{Meerbergen2001} 
and Stewart \cite{Stewart2009}.
Then we demonstrate the efficacy of the proposed shift-invert Lanczos method  
for an example arising in industrial buckling analysis of structures. 

Algorithm \ref{alg:LanFRO} is implemented in MATLAB \cite[p.~120]{Bai2000}.
The full re-orthogonalization is performed.
The accuracy of a computed eigenpair $(\wht{\lam}_i,\wht{x}_i)$ of 
the buckling eigenvalue problem \eqref{eq:prob} is measured by the 
relative residual norm
\begin{align*}
\eta(\wht{\lam}_i,\wht{x}_i) \equiv
	\frac{\|K\wht{x}_i - \wht{\lam}_iK_G\wht{x}_i\|_2}
       	     {(\|K\|_1 + |\wht{\lam}_i|\|K_G\|_1)\|\wht{x}_i\|_2}. 
\end{align*}
The Euclidean angle $\theta_i=\angle(\wht{x}_i, \mc{Z}_c)$ 
is computed for checking if $\wht{x}_i$ is perpendicular to 
the common nullspace $\mc{Z}_c$ of $K$ and $K_G$ \cite{Fraysse1998,Higham1998}.


\begin{example}{\rm~
Let us consider the following matrix 
pair $(K, K_G)$ similar to the ones constructed by 
Meerbergen \cite{Meerbergen2001} and Stewart \cite{Stewart2009}: 
\begin{align*}  
K = Q\Lam Q^T\in\R^{n\times n}
\quad\mbox{and}\quad
K_G = Q\Phi Q^T\in\R^{n\times n},
\end{align*}
where $Q\in\R^{n\times n}$ is a random orthogonal matrix, 
$\Lam\in\R^{n\times n}$ and $\Phi\in\R^{n\times n}$ are
diagonal matrices with diagonal elements 
\[
\Lam_{kk} = 
\begin{cases}
k,\quad\mbox{if}\quad 1\leq k \leq n-m  \\
0,\quad\mbox{otherwise}
\end{cases}
\quad\mbox{and}\quad
\Phi_{kk} = (-1)^{k},\quad 1\leq k \leq n.
\]
By construction, $K$ is positive semi-definite
and $K_G$ is indefinite, and 
the pencil $K - \lam K_G$ is regular.
The last $m$ columns of $Q$ form a basis of the nullspace $\mc{N}(K)$.
For $1\leq k\leq n-m$, 
the $k$-th column of $Q$ is an eigenvector 
and the associated eigenvalue is $\lam_k=(-1)^k\cdot k$.
The zero eigenvalue of $C\equiv (K-\sigma K_G)^{-1}K$ 
is a well-separated eigenvalue, and the associated eigenspace
is also the nullspace of $K$.
We use the MATLAB function \verb|ldl| to compute 
the \ldlt factorization of the shifted matrix $K-\sigma K_G$. 

\begin{figure}[t]
\centering
\includegraphics[width=0.32\textwidth]{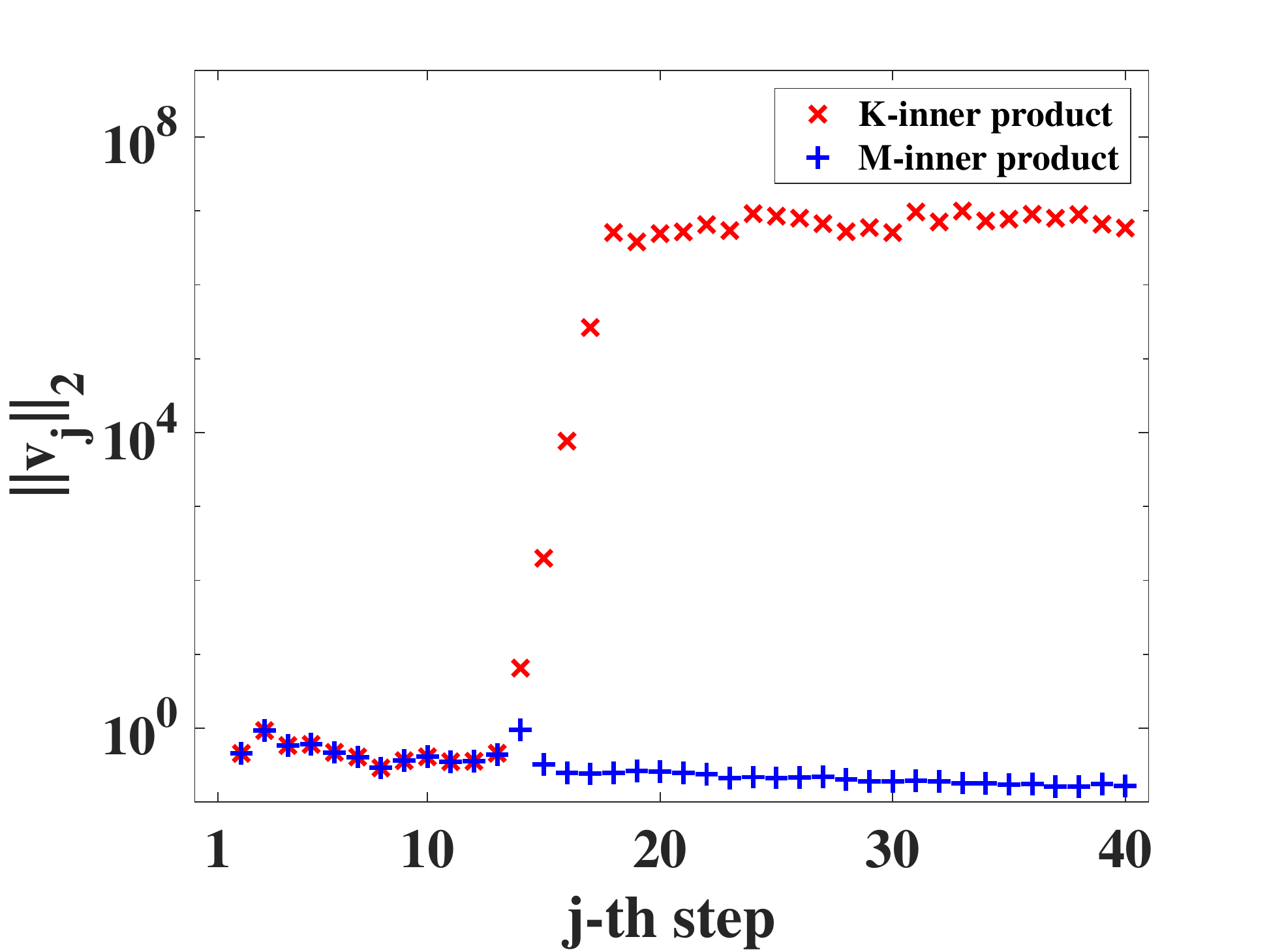}
\includegraphics[width=0.32\textwidth]{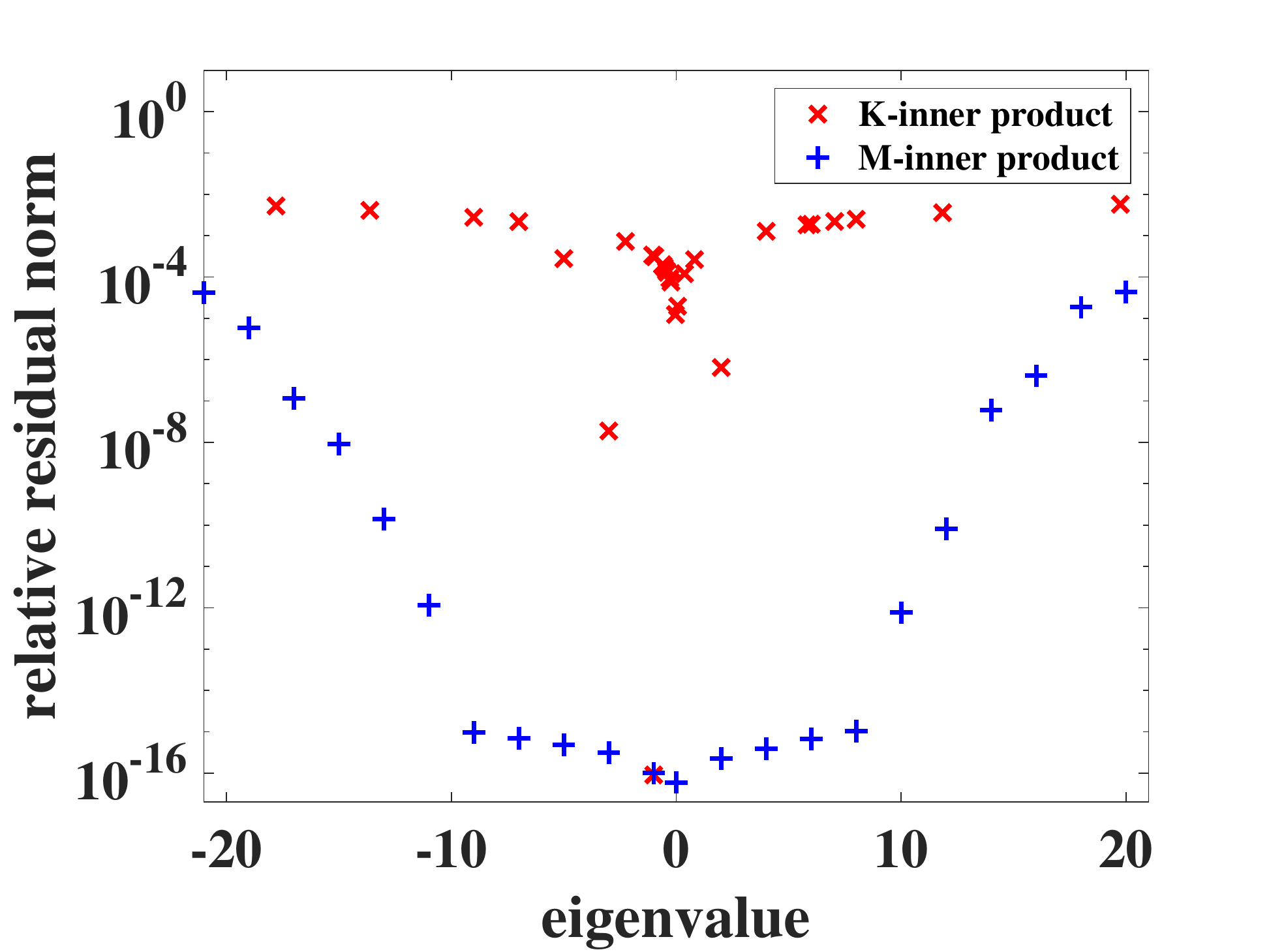}
\includegraphics[width=0.32\textwidth]{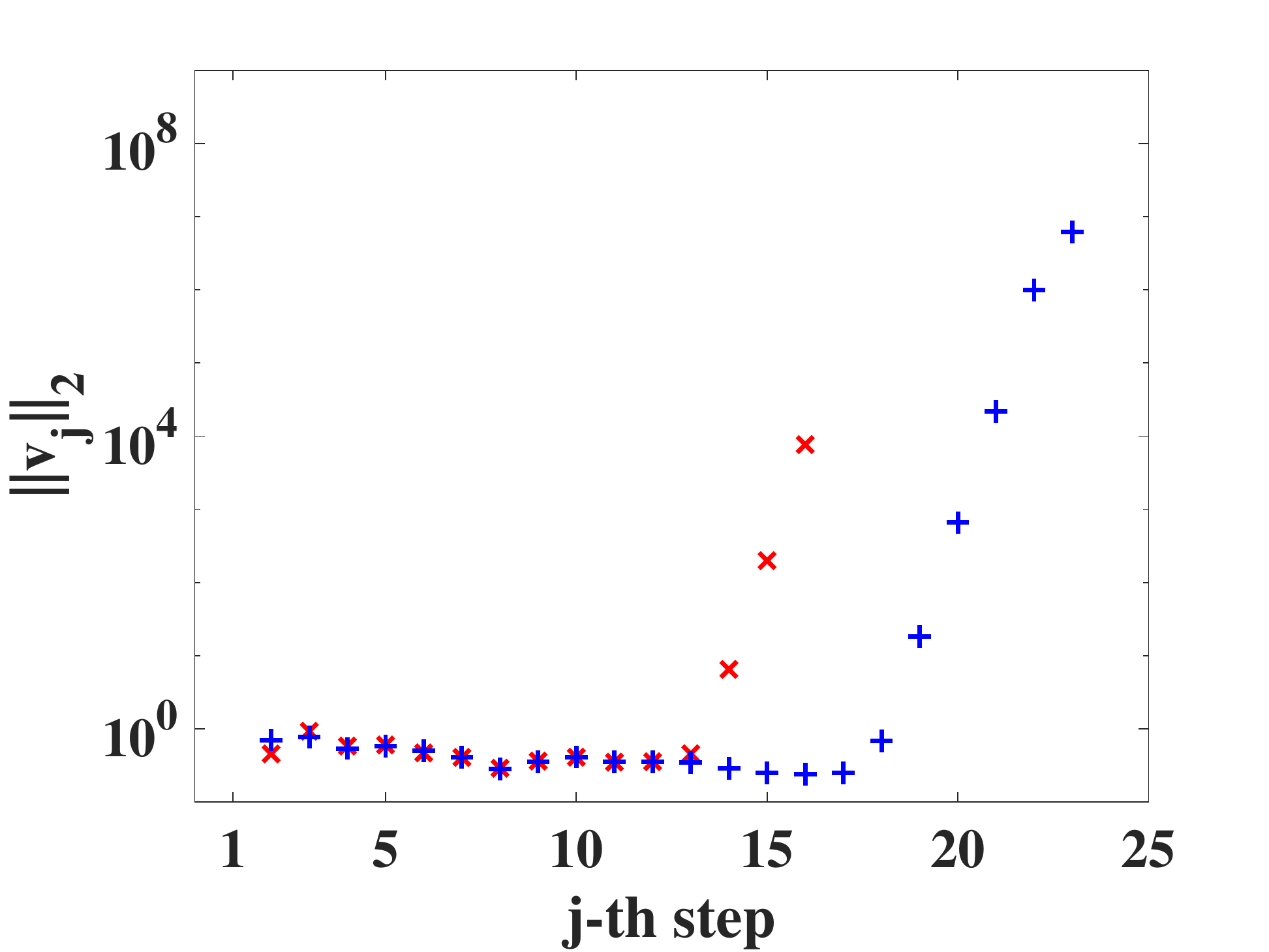}
\caption{
Left: the $2$-norms of the Lanczos vectors $v_j$.
Middle: the relative residual norms of the approximate eigenpairs
$(\wht{\lam}_i,\wht{x}_i)$.
Right: the $2$-norms of the Lanczos vectors $v_j$
with (+) and without (x) implicit restart.
}
\label{fig:InnerProduct}
\end{figure}

For numerical experiments, 
we take $n=500$ and $m=1$. 
We use the buckling spectral transformation 
\eqref{eq:bucklingTr} with the shift $\sigma=-0.6$. 
We run the Lanczos method with $K$-inner product, and 
the starting vector $Cx_0$ with $x_0=[1,\ \ldots,\ 1]^T$.
The approximate eigenpairs $(\wht{\lam}_i,\wht{x}_i)$ of 
\eqref{eq:prob} are computed by 
$(\wht{\lam}_i,\wht{x}_i)=(\frac{\sigma\wht{\mu}_i}{\wht{\mu}_i-1},\wht{x}_i)$.

The left plot of Figure \ref{fig:InnerProduct}
shows the $2$-norms of 40 Lanczos vectors $v_j$. 
As observed by Meerbergen \cite{Meerbergen2001} and Stewart \cite{Stewart2009},
the $2$-norms of Lanczos vectors $v_j$ grows rapidly.  
Consequently, as shown in 
the middle plot of Figure \ref{fig:InnerProduct}, 
the accuracy of approximate eigenpairs 
$(\wht{\lam}_i,\wht{x}_i)$ deteriorates. 
In contrast, when we replace the $K$-inner product by 
the positive definite $M$-inner product with $H_N=I_{m}$. 
We observe that, the $2$-norms of the Lanczos vectors are well bounded.
Multiple eigenvalues near the shift $\sigma$ 
are computed with the relative residual norms around 
the machine precision. 

We note that in \cite{Meerbergen2001}, Meerbergen proposed to control 
the norms of the Lanczos vectors by applying implicit restart.
We experimented the scheme of implicit restart 
at $16$-th iteration of the Lanczos method.
The results are shown in the right plot of Figure \ref{fig:InnerProduct}.
We can see that the $2$-norms of the Lanczos vectors
with and without implicit restart grows rapidly.   
}\end{example}


\begin{example}{\rm ~ 

This is an example from the buckling analysis 
of a finite element model of an airplane shown in Figure~\ref{fig:plane}.
The size of the pencil $K-\lambda K_G$ is $n=67,512$.
The stiffness matrix $K$ is positive semi-definite
and the dimension of the nullspace $\mc{N}(K)$ is known to be $6$, 
which corresponds to the $6$ rigid body modes \cite{Farhat1998}.
The basis $Z$ of $\mc{N}(K)$ is computed by the 
Gaussian-based method \cite{Farhat1998}.
The dimension of the common nullspace $\mc{Z}_c$ of $K$ and $K_G$ 
is $3$, which can be easily computed from the basis
$Z$, see \cite[Theorem 6.4.1]{Golub2013}.
The accuracy of the bases is shown in the table in Figure \ref{fig:plane}.
We are interested in computing the nonzero eigenvalues of
the pencil $K-\lam K_G$ in an interval around zero
and the associated eigenvectors perpendicular to 
the common nullspace $\mc{Z}_c$.

\begin{figure}
\begin{center}
\includegraphics[width=0.21\textwidth,angle=-90,trim=82mm 85mm 85mm 83mm,clip]
{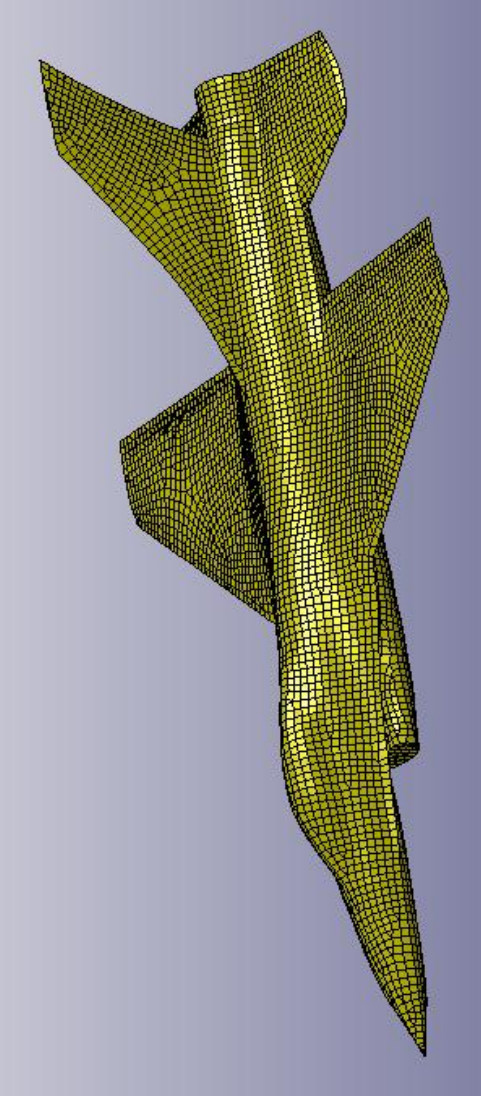} 
\hskip 0.5em 
\raisebox{-50pt}{
\begin{tabular}{| c || c | c| c|} \hline
$i$  &  $d_i/\|K_G\|_1$   &  $\frac{\|Kz_i\|_2}{\|K\|_1\|z_i\|_2}$  &  
$\frac{\|K_Gz_i\|_2}{\|K_G\|_1\|z_i\|_2}$  \\ \hline \hline
$1$ & $6.90\cdot 10^{-5}$ &  $2.74\cdot 10^{-16}$  & $6.78\cdot 10^{-5}$  \\
$2$ & $3.25\cdot 10^{-5}$ &  $4.88\cdot 10^{-16}$  &  $9.06\cdot 10^{-6}$  \\
$3$ & $2.32\cdot 10^{-5}$ &  $4.71\cdot 10^{-16}$  &  $1.19\cdot 10^{-5}$  \\ 
\hline
$4$ & $7.32\cdot 10^{-16}$ &  $2.68\cdot 10^{-17}$  &  $5.01\cdot 10^{-18}$\\
$5$ & $1.26\cdot 10^{-16}$ &  $1.90\cdot 10^{-17}$  &  $4.89\cdot 10^{-18}$  \\
$6$ & $7.81\cdot 10^{-18}$ &  $2.37\cdot 10^{-17}$  &  $5.00\cdot 10^{-18}$  \\ 
\hline
\end{tabular}
}
\caption{Left: Finite element model of an airplane.
Right: Accuracy of the bases for the nullspace of $K$
and common nullspace of $K$ and $K_G$, where
the second column shows the singular values $d_i$ of 
$K_GY$ with $Y$ being an orthonormal basis of $\mc{N}(K)$.
The third and fourth columns show the accuracy of 
the basis $Z=[Z_N\ Z_C]=[z_1\ z_2\ \ldots\ z_6]$.
} \label{fig:plane} 
\end{center}
\end{figure}

We used two methods for computing the matrix-vector product $u=Cv$
described in \S\ref{sec:issues}. For Method 2,  
we determine the permutation matrix $P$ by
maximizing the number of non-zero entries 
in the last $n_3$ columns of $S$ in \eqref{eq:rrp}. 
The MATLAB function \verb|ldl|, which uses 
MA57 \cite{Duff2004} for real sparse matrices,
is used to compute the sparse \ldlt factorization of 
the augmented matrix $A_{\sigma}$ and the submatrix $S^{\sigma}_{11}$.
{The pivot tolerance $\tau=0.1$ is used to
control the numerical stability of the factorization.}
In defining the positive definite matrix $M$, 
we use $H_N= \omega D_N$ and $H_C= \omega I_{n_3}$,
where $D_N$ is a diagonal matrix to 
normalize each column of the matrix $K_GZ_N$ and 
$\omega = \|K\|_1$. 
The starting vector of the Lanczos procedure is $v=Cx_0$ 
with $x_0$ being a random vector.  

To monitor the progress of the Lanczos method,
an approximate eigenpair $(\wht{\mu}_i,\wht{x}_i)$ 
computed from an eigenpair $(\wht{\mu}_i,\wht{s}_i)$ of 
the reduced matrix $T_j$  
is considered to have converged 
if the following two  conditions are satisfied:
\begin{align}  \label{eq:criterion}
|\wht{\mu}_i|\geq tol 
\quad\mbox{and}\quad
\frac{|\sigma|}{(\wht{\mu}_i-1)^2}|\beta_{j}||e_j^T\wht{s}_i| < tol,
\end{align}
where the first condition is used to exclude zero eigenvalues 
and $tol$ is a prescribed tolerance 
(see \cite{Ericsson1980,Grimes1994} and \cite[p.~357]{Parlett1998}). 
In this numerical example, we use the tolerance $tol=10^{-6}$.

We now show the numerical results for computing nonzero eigenvalues of 
the pencil $K-\lambda K_G$ and corresponding eigenvectors  
perpendicular to the common nullspace $\mc{Z}_c$ in the interval $(-8,8)$. 
First, let us consider the left-half interval $(-8,0)$.  With the shift $\sigma = -4.0$, 
the shift-invert Lanczos method (Algorithm 1) 
computed 12 eigenvalues to the machine precision 
in the interval $(-8,0)$ at $38$-th iteration with
either method for the matrix-vector product $u = Cv$. 
The accuracy of the computed eigenpairs 
$(\wht{\lam}_i=\frac{\sigma\wht{\mu}_i}{\wht{\mu}_i-1},\wht{x}_i)$
are shown in Table \ref{tab:results1}.
To validate the number of eigenvalues in the interval $(-8,0)$,
we use the counting scheme described in \S\ref{sec:inertias}.
Using the inertias of the augmented matrix $A_{\alpha}$ with $\alpha = -8$,
by Theorem \ref{thm:rrp}, we have
\[
n(-8,0) = \nu_{-}(A_{\alpha})-\dim(\mc{Z}_c)-\nu_{-}(Z_N^TK_GZ_N)
= 18 - 3 - 3 = 12.
\]
This matches the number of eigenvalues found in the interval.
Alternatively, by using
the inertias of the submatrix $S^{\alpha}_{11}$ with $\alpha = -8$
and Theorem \ref{thm:rrp}, we have
\[
n(-8,0) = \nu_{-}(S^{\alpha}_{11}) - \nu_{-}(Z_N^TK_GZ_N) = 15 - 3 = 12.
\]
This also matches the number of computed eigenvalues in the interval.

\begin{table}[t]
\caption{
Results of 12 computed eigenvalues in the interval $(-8,0)$ after
38 steps of the Lanczos method with the shift $\sigma=-4.0$.
For the 3rd and 4th columns, 
with $\wht{X}\equiv[\wht{x}_1\ \ldots\ \wht{x}_{12}]$,
$\|\wht{X}^TM\wht{X}-I_{12}\|_F=3.82\cdot 10^{-12}$
and the matrix-vector product $u=Cv$ 
is computed by Method 1.
For the 5th and 6th columns,
$\|\wht{X}^TM\wht{X}-I_{12}\|_F=4.55\cdot 10^{-12}$
and the matrix-vector product $u=Cv$ 
is computed by Method 2.
} \label{tab:results1}
\centering
\begin{tabular}{| c || c | c c | c c |}
\hline
$i$  &  $\wht{\lam}_i$  &  
$\eta(\wht{\lam}_i, \wht{x}_i)$  &
$\cos\angle(\wht{x}_i,\mc{Z}_c)$ &
$\eta(\wht{\lam}_i, \wht{x}_i)$  &
$\cos\angle(\wht{x}_i,\mc{Z}_c)$ \\
\hline \hline
$1$   &  $-2.716598$  &  $1.35\cdot 10^{-17}$  &  $2.47\cdot 10^{-17}$  
				  &  $1.47\cdot 10^{-17}$  &  $1.63\cdot 10^{-17}$  \\
$2$   &  $-2.883589$  &  $1.44\cdot 10^{-17}$  &  $6.32\cdot 10^{-17}$  
				  &  $2.47\cdot 10^{-17}$  &  $4.31\cdot 10^{-17}$  \\
$3$   &  $-3.292700$  &  $9.49\cdot 10^{-18}$  &  $5.52\cdot 10^{-17}$  
				  &  $1.14\cdot 10^{-17}$  &  $3.66\cdot 10^{-17}$  \\
$4$   &  $-3.378406$  &  $1.08\cdot 10^{-17}$  &  $1.03\cdot 10^{-17}$  
				  &  $1.20\cdot 10^{-17}$  &  $3.21\cdot 10^{-17}$  \\
$5$   &  $-5.754628$  &  $2.36\cdot 10^{-17}$  &  $5.81\cdot 10^{-17}$  
				  &  $2.11\cdot 10^{-17}$  &  $5.43\cdot 10^{-17}$  \\
$6$   &  $-5.854071$  &  $2.19\cdot 10^{-17}$  &  $1.23\cdot 10^{-16}$  
				  &  $2.54\cdot 10^{-17}$  &  $2.36\cdot 10^{-17}$  \\
$7$   &  $-6.089281$  &  $6.16\cdot 10^{-17}$  &  $1.44\cdot 10^{-16}$  
				  &  $2.17\cdot 10^{-17}$  &  $2.12\cdot 10^{-17}$  \\
$8$   &  $-6.228974$  &  $3.40\cdot 10^{-17}$  &  $8.58\cdot 10^{-17}$  
				  &  $2.06\cdot 10^{-17}$  &  $2.29\cdot 10^{-17}$  \\
$9$   &  $-6.784766$  &  $1.91\cdot 10^{-15}$  &  $7.52\cdot 10^{-17}$  
				  &  $8.37\cdot 10^{-16}$  &  $1.67\cdot 10^{-17}$  \\
$10$  &  $-6.886759$  &  $5.61\cdot 10^{-15}$  &  $5.07\cdot 10^{-17}$  
				   &  $2.88\cdot 10^{-15}$  &  $5.43\cdot 10^{-17}$  \\
$11$  &  $-7.561377$  &  $1.94\cdot 10^{-12}$  &  $2.70\cdot 10^{-16}$  
				   &  $1.87\cdot 10^{-12}$  &  $7.41\cdot 10^{-17}$  \\
$12$  &  $-7.745144$  &  $3.87\cdot 10^{-12}$  &  $1.26\cdot 10^{-16}$  
				   &  $3.82\cdot 10^{-12}$  &  $1.28\cdot 10^{-16}$  \\
\hline
\end{tabular}	
\end{table}

Next let us consider the right-half interval $(0,8)$. 
In this case, we use the shift $\sigma = 4.0$. 
By the shift-invert Lanczos method (Algorithm~1), we found 
13 eigenvalues to the machine precision
in the interval $(0,8)$ at $44$-th iteration with
either method for the matrix-vector product $u = Cv$.
The accuracy of the computed eigenpairs
$(\wht{\lam}_i=\frac{\sigma\wht{\mu}_i}{\wht{\mu}_i-1},\wht{x}_i)$
are shown in Table \ref{tab:results2}.
To validate the number of eigenvalues in the interval $(0,8)$,
we again use the counting scheme described in \S\ref{sec:inertias}.
Using the inertias of the augmented matrix $A_{\alpha}$ with $\alpha = 8$,
by Theorem \ref{thm:rrp}, we have
\[
n(0,8) = \nu_{-}(A_{\alpha})-\dim(\mc{Z}_c)-\nu_{+}(Z_N^TK_GZ_N)
= 16 - 3 - 0 = 13.
\]
This matches the number of eigenvalues found in the interval.
Alternatively, by using
the inertias of the submatrix $S^{\alpha}_{11}$ with $\alpha = 8$
and Theorem \ref{thm:rrp}, we have
\[
n(0,8) = \nu_{-}(S^{\alpha}_{11}) - \nu_{+}(Z_N^TK_GZ_N) = 13 - 0 = 13.
\]
This also matches the number of computed eigenvalues in the interval.

\begin{table}[t]
\caption{
Results of 13 computed eigenvalues in the interval $(0,8)$ after
44 steps of the Lanczos method with the shift $\sigma=4.0$.
For the 3rd and 4th columns, 
with $\wht{X}\equiv[\wht{x}_1\ \ldots\ \wht{x}_{13}]$,
$\|\wht{X}^TM\wht{X}-I_{13}\|_F=1.63\cdot 10^{-11}$
and the matrix-vector product $u=Cv$ 
is computed by Method 1.
For the 5th and 6th columns,
$\|\wht{X}^TM\wht{X}-I_{13}\|_F=1.23\cdot 10^{-11}$
and the matrix-vector product $u=Cv$ 
is computed by Method 2. 
} 
\label{tab:results2}
\centering
\begin{tabular}{| c || c | c c | c c |}
\hline
$i$  &  $\wht{\lam}_i$  &  
$\eta(\wht{\lam}_i, \wht{x}_i)$  &
$\cos\angle(\wht{x}_i,\mc{Z}_c)$ &
$\eta(\wht{\lam}_i, \wht{x}_i)$  &
$\cos\angle(\wht{x}_i,\mc{Z}_c)$ \\
\hline\hline
$1$   &  $2.967043$  &  $1.95\cdot 10^{-17}$  &  $3.01\cdot 10^{-17}$  
				 &  $3.88\cdot 10^{-17}$  &  $3.68\cdot 10^{-17}$  \\
$2$   &  $3.025965$  &  $2.64\cdot 10^{-17}$  &  $2.00\cdot 10^{-16}$  
				 &  $4.05\cdot 10^{-17}$  &  $1.46\cdot 10^{-16}$  \\
$3$   &  $3.917831$  &  $1.96\cdot 10^{-17}$  &  $1.84\cdot 10^{-16}$  
				 &  $1.60\cdot 10^{-17}$  &  $1.29\cdot 10^{-16}$  \\
$4$   &  $4.008941$  &  $2.09\cdot 10^{-17}$  &  $2.76\cdot 10^{-16}$  
				 &  $1.60\cdot 10^{-17}$  &  $1.70\cdot 10^{-16}$  \\
$5$   &  $4.591063$  &  $2.31\cdot 10^{-17}$  &  $3.09\cdot 10^{-16}$  
				 &  $2.36\cdot 10^{-17}$  &  $7.40\cdot 10^{-18}$  \\
$6$   &  $4.662575$  &  $2.60\cdot 10^{-17}$  &  $5.19\cdot 10^{-17}$  
				 &  $2.73\cdot 10^{-17}$  &  $9.46\cdot 10^{-17}$  \\
$7$   &  $5.699271$  &  $4.59\cdot 10^{-17}$  &  $5.61\cdot 10^{-16}$  
				 &  $4.15\cdot 10^{-17}$  &  $6.26\cdot 10^{-17}$  \\
$8$   &  $5.725937$  &  $4.70\cdot 10^{-17}$  &  $7.27\cdot 10^{-17}$  
				 &  $5.67\cdot 10^{-17}$  &  $1.16\cdot 10^{-16}$  \\
$9$   &  $6.465175$  &  $8.05\cdot 10^{-17}$  &  $1.08\cdot 10^{-16}$  
				 &  $6.01\cdot 10^{-16}$  &  $1.14\cdot 10^{-16}$  \\
$10$  &  $6.598173$  &  $1.34\cdot 10^{-15}$  &  $9.94\cdot 10^{-16}$  
				  &  $6.85\cdot 10^{-15}$  &  $3.46\cdot 10^{-16}$ \\
$11$  &  $7.285975$  &  $4.39\cdot 10^{-15}$  &  $3.17\cdot 10^{-15}$  
				  &  $5.24\cdot 10^{-15}$  &  $9.52\cdot 10^{-16}$ \\		
$12$  &  $7.626265$  &  $2.58\cdot 10^{-14}$  &  $6.71\cdot 10^{-15}$  
				  &  $2.72\cdot 10^{-14}$  &  $3.44\cdot 10^{-15}$ \\	  
$13$  &  $7.880296$  &  $1.21\cdot 10^{-12}$  &  $3.01\cdot 10^{-14}$  
				  &  $1.26\cdot 10^{-12}$  &  $2.98\cdot 10^{-14}$ \\
\hline
\end{tabular}
\end{table}

We observed a significant difference in 
the numbers of the non-zero entries of the 
triangular factor $L$ in the \verb|ldl| factorizations of
the augmented matrix $A_{\sigma}$ and the submatrix $S_{11}^{\sigma}$.  
For example, with the shift $\sigma = -4.0$, 
the number of non-zero entries of $L$ 
from the submatrix $S_{11}^{\sigma}$ is 
$15,142,866$, which is $65.5\%$ less than 
the number of non-zero entries of $L$ 
from the augmented matrix $A_{\sigma}$, 
which is $43,940,581$, see Figure~\ref{fig:fill-in}. 
Similar results are also observed with the shift $\sigma = 4.0$. 
Hence we strongly advocate the use of Method 2 for 
computing the matrix-vector product $u = Cv$. 

\begin{figure}[t]
\centering
\includegraphics[width=0.35\textwidth,trim=0mm 0mm 0mm 0mm,clip=true]
{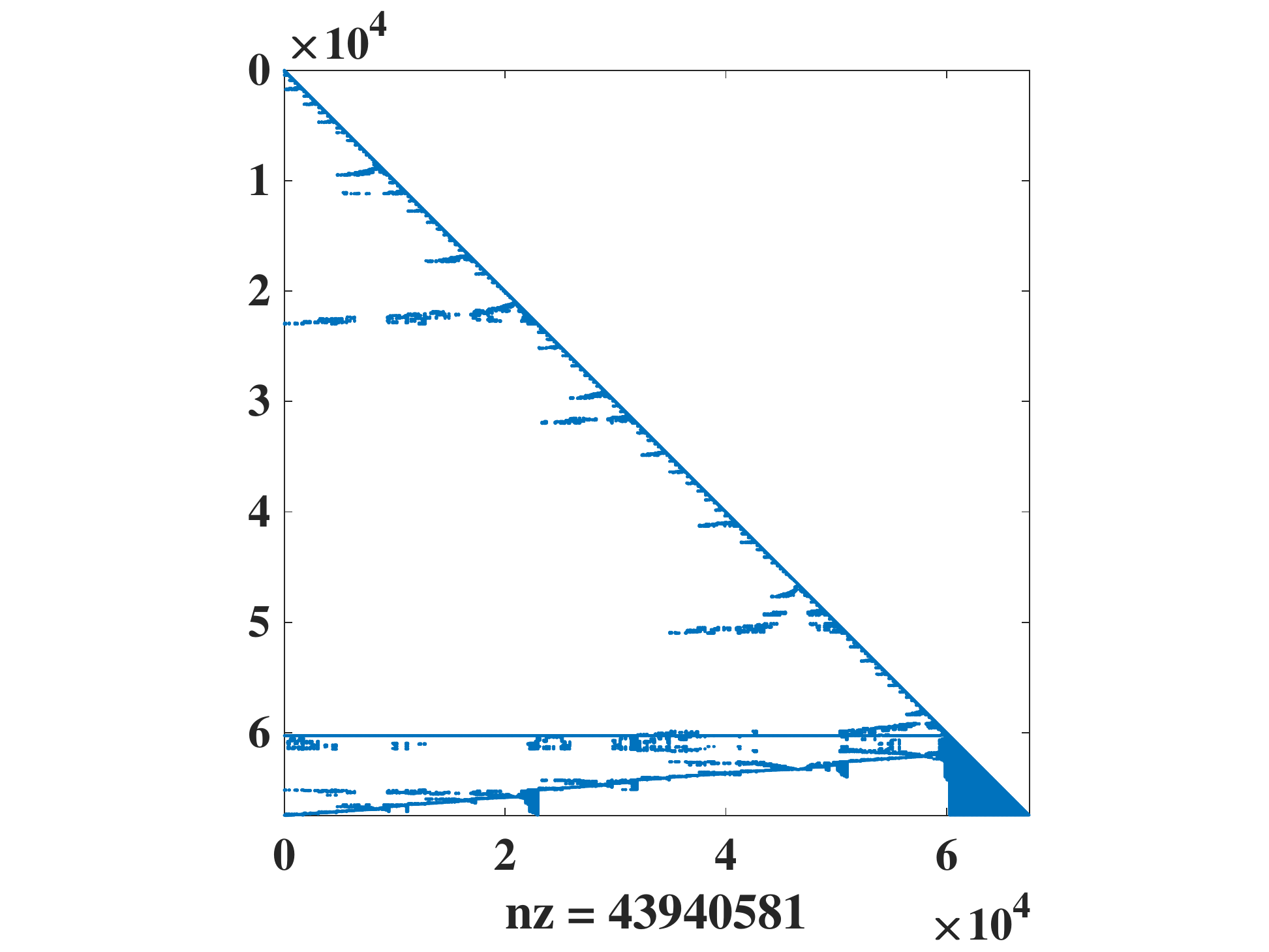}
\includegraphics[width=0.35\textwidth,trim=0mm 0mm 0mm 0mm,clip=true]
{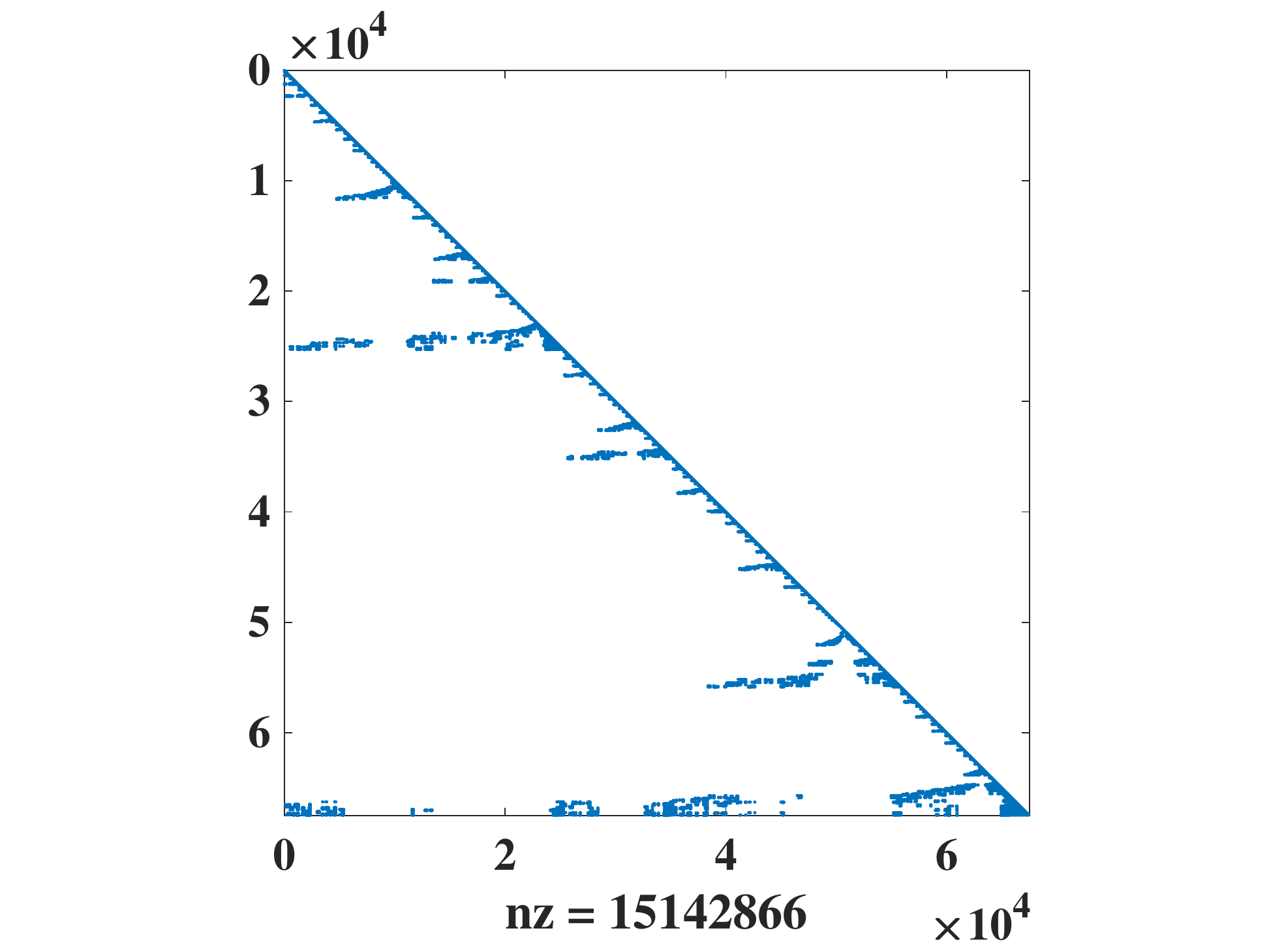}
\caption{The sparsity pattern of the matrices $L$ 
from the augmented matrix $A_{\sigma}$ (left) 
and the submatrix $S_{11}^{\sigma}$ (right).
The shift is $\sigma=-4.0$ and the pivot tolerance is $\tau=0.1$.}
\label{fig:fill-in}
\end{figure}

}
\end{example} 


\section{Concluding Remark}  \label{sec:conclusion}
We studied the buckling eigenvalue problem  of 
singular pencil, and addressed the two open issues 
associated with the shift-invert Lanczos method.
We found that the proposed scheme for counting the 
number of eigenvalues is a reliable tool for 
the validation.


\section*{Acknowledgments}
This work was supported in part by the U.S. National Science Foundation 
under Award DMS-1913364.

\appendix

\section{Canonical form of a symmetric semi-definite pencil $A - \lam B$} 
\label{appendix:can}

In this section, we give a constructive derivation of 
a canonical form of a symmetric semi-definite pencil $A - \lam B$,
namely $A$ is symmetric and $B$ is symmetric semi-positive definite.

\begin{theorem} \label{thm:canform}
For a symmetric semi-definite pencil $A-\lam B$, 
there exists a non-singular matrix $W\in \R^{n\times n}$ such that
\begin{equation} \label{eq:canonicalform}
W^T A W =  \kbordermatrix{
		& 2n_0 	& n_1       		& n_2		& n_3   \\
2n_0 	& S   	&                   	&          		&	\\
n_1          	&      		& \Lam_1       	&          		&	\\
n_2	 	&		&                   	& \Lam_2 	&	\\
n_3 		&  		&                   	&                     	& 0}
\quad \mbox{and} \quad
W^T B W =  \kbordermatrix{
		& 2n_0 		& n_1       	& n_2		& n_3   \\
2n_0         & \Omega  	&             	&    			&	\\
n_1          	&                   	& I_{n_1}  	&          		&	\\
n_2 		&                   	&             	& 0  			&	\\
n_3		&			&		& 			& 0},
\end{equation}
where
\begin{equation*}
S \equiv I_{n_0}\otimes\twobytwo{0}{1}{1}{0},
\quad
\Omega \equiv I_{n_0}\otimes\twobytwo{1}{0}{0}{0},
\end{equation*}
$\Lam_1$ and $\Lam_2$ are diagonal matrices with real diagonal entries, 
and $\Lam_2$ is non-singular. 
Moreover, we have 
\begin{align*}
n_0   & = \dim(\mc{N}(B)) - n_2 - n_3,  \\
n_1	& = \rank(B) - n_0,  \\
n_2 	& = \rank(\mc{P}_{\mc{N}(B)}A\mc{P}_{\mc{N}(B)}),  \\
n_3	& = \dim(\mc{N}(A)\cap\mc{N}(B)), 
\end{align*}
where $\mc{P}_{\mc{N}(B)}$ is the orthogonal projection onto $\mc{N}(B)$.
\end{theorem}

We first introduce the following lemma due to Fix 
and Heiberger \cite{Fix1972}, also see \cite[Sec. 15.5]{Parlett1998}.

\begin{lemma}  \label{lemma:FixHeiberger}
For the symmetric semi-definite pencil $A - \lam B$,
there exists a non-singular matrix $W\in \R^{n\times n}$ such that
\begin{align*}  
W^TAW = 
\kbordermatrix{
		&  n_0		&  n_1		&  n_2		&  n_0	&  n_3		\\
n_0		&  A_{00}		&  A_{01}		&  A_{02}		&  \Sigma	&  0		\\
n_1		&  A_{01}^T	&  A_{11}		&  A_{12}		&		&		\\
n_2		&  A_{02}^T	&  A_{12}^T	&  \Lam_2		&		&		\\
n_0		&  \Sigma		&			&			&  0		&		\\
n_3		&  0			&			&			&		&  0		
}
\quad\mbox{and}\quad
W^TBW = 
\kbordermatrix{
		&  n_0		&  n_1		&  n_2		&  n_0	&  n_3		\\
n_0		&  I_{n_0}		&			&			&		&		\\
n_1		&			&  I_{n_1}		&			&		&		\\
n_2		&			&			&  0			&		&		\\
n_0		&			&			&			&  0		&  		\\
n_3		&			&			&			&		&  0	
},
\end{align*}
where $\Lam_2$ and $\Sigma$ are non-singular, diagonal matrices 
with real diagonal entries.
\end{lemma}

\noindent {\em Proof of Theorem \ref{thm:canform}.} 
By Lemma \ref{lemma:FixHeiberger},
there exists a non-singular matrix $W_0\in \R^{n\times n}$ such that
\begin{align*}  
A^{(1)} \equiv W_0^TAW_0 = 
\kbordermatrix{
		&  n_0		&  n_1		&  n_2		&  n_0	&  n_3		\\
n_0		&  A_{00}		&  A_{01}		&  A_{02}		&  \Sigma	&  0		\\
n_1		&  A_{01}^T	&  A_{11}		&  A_{12}		&		&		\\
n_2		&  A_{02}^T	&  A_{12}^T	&  \Lam_2		&		&		\\
n_0		&  \Sigma		&			&			&  0		&		\\
n_3		&  0			&			&			&		&  0		
}
\quad\mbox{and}\quad
B^{(1)} \equiv W_0^TBW_0 = 
\kbordermatrix{
		&  n_0		&  n_1		&  n_2		&  n_0	&  n_3		\\
n_0		&  I_{n_0}		&			&			&		&		\\
n_1		&			&  I_{n_1}		&			&		&		\\
n_2		&			&			&  0			&		&		\\
n_0		&			&			&			&  0		&  		\\
n_3		&			&			&			&		&  0	
},
\end{align*}
where $\Lam_2$ and $\Sigma$ are non-singular, diagonal matrices
with real diagonal entries.

Let
\begin{align*}
W_1 \equiv 
\kbordermatrix{
		&  n_0		&  n_1		&  n_2		&  n_0	&  n_3	\\
n_0		&  I_{n_0}		&			&			&		&		\\
n_1		&			&  I_{n_1}		&			&		&		\\
n_2		&			&			&  I_{n_2}		&		&		\\
n_0	&  -\Sigma^{-1}A_{00}/2  &  -\Sigma^{-1}A_{01}  &  -\Sigma^{-1}A_{02}   &  I_{n_0}  &  \\
n_3		&			&			&   			&		&  I_{n_3}		
},
\end{align*}
then
\begin{align*}
A^{(2)}\equiv W_1^TA^{(1)}W_1 = 
\kbordermatrix{
		&  n_0		&  n_1		&  n_2		&  n_0	&  n_3	\\
n_0		&  0			&  			&  			&  \Sigma	&  		\\
n_1		&  			&  A_{11}		&  A_{12}		&		&		\\
n_2		&  			&  A_{12}^T	&  \Lam_2		&		&		\\
n_0		&  \Sigma		&			&			&  0		&		\\
n_3		&  			&			&			&		&  0		
}
\quad\mbox{and}\quad
B^{(2)}\equiv W_1^TB^{(1)}W_1 = 
\kbordermatrix{
		&  n_0		&  n_1		&  n_2		&  n_0	&  n_3	\\
n_0		&  I_{n_0}		&			&			&		&		\\
n_1		&			&  I_{n_1}		&			&		&		\\
n_2		&			&			&  0			&		&		\\
n_0		&			&			&			&  0		&  		\\
n_3		&			&			&			&		&  0	
}.
\end{align*}
Next let
\begin{align*}  
W_2 \equiv 
\kbordermatrix{
		&  n_0		&  n_1		&  n_2		&  n_0	&  n_3	\\
n_0		&  I_{n_0}		&			&			&		&		\\
n_1		&			&  I_{n_1}		&			&		&		\\
n_2		&&  -\Lam_{2}^{-1}A_{12}^T	&  I_{n_2}		&		&		\\
n_0		&			&   			&  			&  I_{n_0}	&		\\
n_3		&			&			&   			&		&  I_{n_3}		
},
\end{align*}
then
\begin{align*}
A^{(3)}\equiv W_2^TA^{(2)}W_2 = 
\kbordermatrix{
		&  n_0		&  n_1		&  n_2		&  n_0	&  n_3	\\
n_0		&  0			&  			&  			&  \Sigma	&  		\\
n_1		&  			&  C_{11}		&  			&		&		\\
n_2		&  			&  			&  \Lam_2	&		&		\\
n_0		&  \Sigma		&			&			&  0		&		\\
n_3		&  			&			&			&		&  0		
}
\quad\mbox{and}\quad
B^{(3)}\equiv W_2^TB^{(2)}W_2 = 
\kbordermatrix{
		&  n_0		&  n_1		&  n_2		&  n_0	&  n_3	\\
n_0		&  I_{n_0}		&			&			&		&		\\
n_1		&			&  I_{n_1}		&			&		&		\\
n_2		&			&			&  0			&		&		\\
n_0		&			&			&			&  0		&  		\\
n_3		&			&			&			&		&  0	
},
\end{align*}
where $C_{11}\in\R^{n_1 \times n_1}$ is symmetric and
$C_{11} = A_{11} - A_{12}\Lam_{2}^{-1}A_{12}^T$.

Define the permutation matrix 
\begin{align*}
P_3 \equiv
\left[ \begin{array}{ccccc}
I_{n_0} & 0 & 0 & 0 & 0  \\
0 & 0 & I_{n_1} & 0 & 0  \\
0 & 0 & 0   & I_{n_2} & 0  \\
0 & I_{n_0} & 0 & 0 & 0  \\
0 & 0 & 0 & 0 & I_{n_3}
\end{array} \right],
\end{align*}
then
\begin{align*}
A^{(4)}\equiv P_3^TA^{(3)}P_3 = 
\kbordermatrix{
		&  n_0		&  n_0		&  n_1		&  n_2		&  n_3	\\
n_0		&  			&  \Sigma		&  			&  			&  		\\
n_0		&  \Sigma		&  			&  			&			&		\\
n_1		&  			&  			&  C_{11}		&			&		\\
n_2		&  			&			&			&  \Lam_2		&		\\
n_3		&  			&			&			&			&  0		
}
\quad\mbox{and}\quad
B^{(4)}\equiv P_3^TB^{(3)}P_3 = 
\kbordermatrix{
		&  n_0		&  n_0		&  n_1		&  n_2		&  n_3	\\
n_0		&  I_{n_0}		&			&			&			&		\\
n_0		&			&  0			&			&			&		\\
n_1		&			&			&  I_{n_1}		&			&		\\
n_2		&			&			&			&  0			&  		\\
n_3		&			&			&			&			&  0	
}.
\end{align*}

\medskip
Since $C_{11}\in\R^{n_1 \times n_1}$ is symmetric, 
it admits the eigen-decomposition
\begin{align*}
C_{11} = Q_1\Lam_1Q_1^T,
\end{align*}
where $Q_1\in\R^{n_1 \times n_1}$ is an orthogonal matrix and 
$\Lam_1\in\R^{n_1 \times n_1}$ is a diagonal matrix.
Applying the congruent transformation associated with 
$W_4\equiv\diag(I_{n_0},\Sigma^{-1},Q_1,I_{n_2},I_{n_3})$, we have
\begin{align*}
A^{(5)}\equiv W_4^TA^{(4)}W_4 = 
\kbordermatrix{
		&  n_0		&  n_0		&  n_1		&  n_2		&  n_3	\\
n_0		&  			&  I_{n_0}		&  			&  			&  		\\
n_0		& I_{n_0}		&  			&  			&			&		\\
n_1		&  			&  			&  \Lam_1		&			&		\\
n_2		&  			&			&			&  \Lam_2		&		\\
n_3		&  			&			&			&			&  0		
}
\quad\mbox{and}\quad
B^{(5)}\equiv W_4^TB^{(4)}W_4 = 
\kbordermatrix{
		&  n_0		&  n_0		&  n_1		&  n_2		&  n_3	\\
n_0		&  I_{n_0}		&			&			&			&		\\
n_0		&			&  0			&			&			&		\\
n_1		&			&			&  I_{n_1}		&			&		\\
n_2		&			&			&			&  0			&  		\\
n_3		&			&			&			&			&  0	
}.
\end{align*}

\medskip
Last, define the permutation matrix $P_5\equiv\diag(E, I_{n_1}, I_{n_2}, I_{n_3})$ 
with $E\equiv [e_1\ e_{n_0+1}\ e_2 \ldots\ e_{2n_0}]$
and we have the canonical form in \eqref{eq:canonicalform}
\begin{align*}
A^{(6)}\equiv P_5^TA^{(5)}P_5 = 
\kbordermatrix{
		&  2n_0		&  n_1		&  n_2		&  n_3	\\
2n_0		&  S			&  			&  			&  		\\
n_1		&  			&  \Lam_1		&			&		\\
n_2		&			&			&  \Lam_2		&		\\
n_3		&			&			&			&  0		
}
\quad\mbox{and}\quad
B^{(6)}\equiv P_5^TB^{(5)}P_5 = 
\kbordermatrix{
		&  2n_0		&  n_1		&  n_2		&  n_3	\\
2n_0		&  \Omega	&			&			&		\\
n_1		&			&  I_{n_1}		&			&		\\
n_2		&			&			&  0			& 	 	\\
n_3		&			&			&			&  0	
},
\end{align*}
where
\begin{equation*}
S \equiv I_{n_0}\otimes\twobytwo{0}{1}{1}{0}
\quad\mbox{and}\quad
\Omega \equiv I_{n_0}\otimes\twobytwo{1}{0}{0}{0}.
\end{equation*}
The canonical form~\eqref{eq:canonicalform} is obtained with 
$W\equiv W_0W_1W_2P_3W_4P_5$. 

Now we interpret the dimension of each block matrix.
From the canonical form of $B$ in Eq.~\eqref{eq:canonicalform}, 
we can infer that $n_0 = \dim(\mc{N}(B)) - n_2 - n_3$ and 
$n_1 = \rank(B) - n_0$. Also, $n_3 = \dim(\mc{N}(A)\cap\mc{N}(B))$.
To interpret $n_2$, 
let $Z\in\R^{n\times(n_0+n_2+n_3)}$ be the basis of $\mc{N}(B)$ 
consisting of the columns of $W$ and consider the QR decomposition of $Z=QR$. 
Since $Q$ is an orthonormal basis of $\mc{N}(B)$,
$\rank(\mc{P}_{\mc{N}(B)}A\mc{P}_{\mc{N}(B)})=\rank(Q^TAQ)$.
By the Sylvester's law,
$\rank(Q^TAQ)=\rank(Z^TAZ)$.
But, from the canonical form~\eqref{eq:canonicalform},
$Z^TAZ=\diag(0_{n_0},\Lam_2,0_{n_3})$ and $\rank(Z^TAZ)=n_2$.
Therefore, $n_2=\rank(\mc{P}_{\mc{N}(B)}A\mc{P}_{\mc{N}(B)})$.
\hfill $\Box$ 

\begin{corollary}  \label{cor:sd} 
The symmetric semi-definite pencil $A-\lam B$ is 
simultaneously diagonalizable if and only if $n_0=0$. 
In this case, we have the canonical form
\begin{equation*} 
W^T A W =  \kbordermatrix{
		& n_1       		& n_2			&  n_3	\\
n_1		& \Lam_1		&          			&		\\	
n_2		&			&  \Lam_2			&  		\\
n_3		&			&				&  0
}
\quad \mbox{and} \quad
W^T B W =  \kbordermatrix{
		& n_1       		& n_2			&  n_3	\\
n_1		& I_{n_1}		&          			&		\\	
n_2		&			&  0				&  		\\
n_3		&			&				&  0
},
\end{equation*}
\end{corollary}
\begin{proof}
From the pairs $(S,\Omega)$ and $(\Lam_2,0)$ in Eq.~\eqref{eq:canonicalform}, 
we note that the algebraic and geometric multiplicity of the infinite eigenvalues 
are $2n_0+n_2$ and $n_0+n_2$, respectively. 
Therefore, the symmetric semi-definite pencil $A-\lam B$ 
is simultaneously diagonalizable if and only if $n_0 = 0$. 
\end{proof}


\bibliographystyle{abbrv}
\bibliography{refs}%

\end{document}